\documentclass[12pt]{article}
\usepackage{graphicx}
\usepackage{amsmath}
\usepackage{amsfonts}
\usepackage{amssymb}
\usepackage{url}
\usepackage{float,color,ulem}

\newfloat{figure}{H}{lof}
\newfloat{table}{H}{lot}
\floatname{figure}{\figurename}
\floatname{table}{\tablename}

\newtheorem{theorem}{Theorem}[section]

\newtheorem{lemma}[theorem]{Lemma}

\numberwithin{equation}{section}

\newenvironment{proof}[1][Proof]{\textbf{#1.} }{\ \rule{0.5em}{0.5em}}

\usepackage{color}

\topmargin=-0.5in \everymath{\displaystyle}
\begin{document}
\baselineskip=18pt

\pagenumbering{arabic}

\begin{center}
{\Large {\bf On a Vector-host Epidemic Model with Spatial Structure}}

\vspace{.3in}

Pierre Magal \textsuperscript{1},
G.F. Webb \textsuperscript{2},
Yixiang Wu \textsuperscript{2}

\vspace{.3in}

\begin{footnotesize}
\textbf{1} Mathematics Department, University of Bordeaux, Bordeaux, France
\\
\textbf{2} Mathematics Department, Vanderbilt University, Nashville, TN\\
\end{footnotesize}

\bigskip

\vspace{0.2in}
\end{center}

\begin{abstract}
In this paper, we study a reaction-diffusion vector-host epidemic model. We define the basic reproduction number $R_0$ and show that $R_0$ is a threshold parameter: if $R_0\le 1$ the disease free steady state is globally stable; if $R_0>1$ the model has a unique globally stable positive steady state. Our proof combines arguments from monotone dynamical system theory, persistence theory, and the theory of asymptotically autonomous semiflows.
\end{abstract}

\medskip


\section{Introduction}
In recent years, many authors (e.g. \cite{allen2008asymptotic, capasso1978global, cui2017spatial, cui2017dynamics, deng2016dynamics, fitzgibbon1994diffusive, fitzgibbon1995diffusion, fitzgibbon2004reaction, fitzgibbon2005modelling, fitzgibbon2008simple, fitzgibbon2017outbreak, kuto2017concentration, lai2014repulsion, li2017varying, lou2011reaction, magal2017spatial, peng2012reaction, vaidya2012avian, wang2015dynamics, wang2016reaction, webb1981reaction, yu2016nonlocal}) propose reaction-diffusion models to study the transmission of diseases. Among them, Fitzgibbon et al. \cite{fitzgibbon2004reaction, fitzgibbon2005modelling} apply a reaction-diffusion system on non-coincident domains to describe the circulation of diseases between two hosts; Lou and Zhao \cite{lou2011reaction} propose a reaction-diffusion model with delay and nonlocal term to study the spread of malaria; Vaidya, Wang and Zou \cite{vaidya2012avian} study the transmission of avian influenza in wild birds by a reaction-diffusion model with spatial heterogeneous coefficients. 

Although diffusive epidemic models are studied extensively, they have only recently been used to describe a real world situation \cite{fitzgibbon2017outbreak, magal2017spatial}. In \cite{magal2017spatial}, we simulate the spatial spread of seasonal influenza in Puerto Rico using a diffusive SIR model based on the geographical and population data. In \cite{fitzgibbon2017outbreak}, the authors demonstrate the effectiveness of a diffusive vector-host epidemic model in understanding the Zika outbreak in Rio De Janeiro.  Our objective in this manuscript is to provide a rigorous analysis of the reaction-diffusion model proposed in  \cite{fitzgibbon2017outbreak}. 

Suppose that individuals are living in a bounded domain $\Omega\subset \mathbb{R}^n$ with smooth boundary $\partial\Omega$. 
Let $H_i(x, t), V_u(x, t)$ and $V_i(x, t)$ be the density of infected hosts, uninfected vectors, and infected vectors at position $x$ and time $t$, respectively. Then the model proposed in \cite{fitzgibbon2017outbreak} to study the outbreak of Zika in Rio De Janerio is the following reaction-diffusion system
\begin{equation}
\left\{
\begin{array}{lll} \label{main}
&\frac{\partial}{\partial t}H_i-\triangledown\cdot \delta_1(x)\triangledown H_i=-\lambda(x) H_i+\sigma_1(x) H_u(x)V_i,\ \ \ &x\in\Omega, t>0,\\
&\frac{\partial}{\partial t}V_u-\triangledown\cdot \delta_2(x)\triangledown V_u=-\sigma_2(x)V_u H_i+\beta(x)(V_u+V_i)-\mu(x) (V_u+V_i)V_u,\ \ \ &x\in\Omega, t>0,\\
&\frac{\partial}{\partial t}V_i-\triangledown\cdot \delta_2(x)\triangledown V_i=\sigma_2(x) V_uH_i-\mu(x)(V_u+V_i)V_i,\ \ \ &x\in\Omega, t>0.
\end{array}
\right.
\end{equation}
with homogeneous Neumann boundary condition
\begin{equation}\label{mainb}
\frac{\partial}{\partial n}H_i= \frac{\partial}{\partial n}V_u= \frac{\partial}{\partial n}V_i=0,\ \ \ x\in\partial\Omega, t>0,
\end{equation}
and initial condition
\begin{equation}\label{maini}
\left(H_i(.,0),V_u(.,  0),V_i(x, 0)\right) =\left( H_{i0},V_{u0},V_{i0}(x) \right) \in C(\bar{\Omega}; \mathbb{R}_+^3).
\end{equation}
where $\delta_1, \delta_2 \in C^{1+ \alpha}(\bar{\Omega})$ are strictly positive, the functions $H_u, \lambda, \beta,\sigma_1, \sigma_2$ and $\mu$ are strictly positive and belong to $C(\bar{\Omega})$. 
The flux of new infected human is given by $\sigma_1(x) H_u(x)V_i(t,x)$ in which $H_u(x)$ is the density of susceptible population depending on the spatial location $x$. The main idea of this model is to assume the susceptible human is (almost) not affected by the epidemic during a relatively short period of time and therefore flux of new infected is remaining (almost) constant. Such a  functional response mainly permit to take care of realistic density of population distributed in space. For Zika in Rio De Janerio the number infected is fairly small in comparison with the number of the all population (less than $1\%$ accordingly to \cite{Bastos2016}). Therefore the density of susceptible can be considered to be constant without been altered by the epidemic.       


In section 2, we define the basic reproductive number $R_0$ as the spectral radius of $-CB^{-1}$, i.e. $R_0=r(-CB^{-1})$, where $B:D(B) \subset C(\bar\Omega; \mathbb{R}^2) \to C(\bar\Omega; \mathbb{R}^2)$ and $C: C(\bar\Omega; \mathbb{R}^2) \to C(\bar\Omega; \mathbb{R}^2)$ are linear operators on $C(\bar\Omega; \mathbb{R}^2)$ with
\begin{equation*}
B=
\begin{pmatrix}
\triangledown\cdot\delta_1\triangledown & 0\\
0 & \triangledown\cdot\delta_2\triangledown
\end{pmatrix}+
 \begin{pmatrix}
-\lambda & \sigma_1 H_u \\
0 & -\mu\hat V
\end{pmatrix}\ \ \ \text{ and }\ \ \ 
C=\begin{pmatrix}
0 & 0 \\
\sigma_2\hat V & 0
\end{pmatrix},
\end{equation*}
with the suitable domain $D(B)$ (see \cite{Stewart80,Mora83}). 
 
The equilibria of \eqref{main}-\eqref{maini} are solutions of the following elliptic system:
\begin{equation}\label{maine}
\left\{
\begin{array}{lll}
-\triangledown\cdot \delta_1(x)\triangledown H_i=-\lambda(x) H_i+\sigma_1(x) H_u(x)V_i,\ \ \ &x\in\Omega, \\
-\triangledown\cdot \delta_2(x)\triangledown V_u=-\sigma_2(x)V_u H_i+\beta(x)(V_u+V_i)-\mu(x) (V_u+V_i)V_u,\ \ \ &x\in\Omega,\\
-\triangledown\cdot \delta_2(x)\triangledown V_i=\sigma_2(x) V_uH_i-\mu(x)(V_u+V_i)V_i,\ \ \ &x\in\Omega,\\
\frac{\partial}{\partial n}H_i= \frac{\partial}{\partial n}V_u= \frac{\partial}{\partial n}V_i=0,\ \ \ &x\in\partial\Omega.
\end{array}
\right.
\end{equation}
The system always has one trivial equilibrium $E_0$ and a unique semi-trivial equilibrium $E_1=(0, \hat V, 0)$. In section 2, we prove that $E_1$ is globally asymptotically stable if $R_0<1$ in Theorem \ref{theorem_Rless}.

Our main result is in section 3, where we show that \eqref{main}-\eqref{maini} has a unique globally asymptotically stable positive steady state $E_2=(\hat H_i, \hat V_u, \hat V_i)$ if $R_0>1$ (see Theorem \ref{theorem_Rbigger}). We remark that it is usually not an easy task to prove the global stability of the positive steady state for a three-equation parabolic system when there is no clear Lyapunov type functional.  Our proof combines arguments from monotone dynamical system theory, persistence theory, and the theory of asymptotically autonomous semiflows. 

We briefly summarize our idea of proof here. Adding up the second and third equations in \eqref{main} and letting $V:=V_u+V_i$, $V$ satisfies the diffusive logistic equation $\partial_ t V_u-\triangledown\cdot \delta_2\triangledown V=\beta V-\mu V^2$. Since this equation has a globally stable positive steady state $\hat V$, it is tempting to assume that the dynamics of \eqref{main}-\eqref{maini} is determined by the limit system
\begin{equation}
\left\{
\begin{array}{lll} \label{main_reducedi}
&\frac{\partial}{\partial t}\tilde H_i-\triangledown\cdot \delta_1\triangledown \tilde H_i=-\lambda \tilde H_i+\sigma_1 \tilde H_u\tilde V_i,\ \ \ &x\in\Omega, t>0,\\
&\frac{\partial}{\partial t}\tilde V_i-\triangledown\cdot \delta_2\triangledown \tilde V_i=\sigma_2 (\hat V-\tilde V_i)^+\tilde H_i-\mu\hat V \tilde V_i,\ \ \ &x\in\Omega, t>0.
\end{array}
\right.
\end{equation}
However even for ordinary differential equation (ODE) systems, Thieme \cite{thieme1994asymptotically} gives many examples where the dynamics of the limit and original systems are quite different. A remedy to this is the theory of asymptotically autonomous semiflows (see \cite[Theorem 4.1]{thieme1992convergence}), which is generalized from the well-known theory by Markus on asymptotically autonomous ODE systems. Applying this theory,  to prove the convergence of $(H_i(\cdot, t), V_i(\cdot, t))$, it suffices to show: (a) system \eqref{main_reducedi} has a  unique positive steady state $(\hat H_i, \hat V_i)$; (b) The steady state $(\hat H_i, \hat V_i)$ of \eqref{main_reducedi} is globally stable in $W:=\{(H_{i0}, V_{i0})\in C(\bar\Omega; \mathbb{R}_+^2):  \ H_{i0}+V_{i0}\neq 0\}$; (c) The  $\omega-$limit set of $(H_i(\cdot, t), V_i(\cdot, t))$ intersects $W$.  The proof of (a) is given in section 3.1.1. The proof of (b) is provided in section 3.1.2, where we take advantage of the monotonicity of \eqref{main_reducedi}. To show (c), we use the uniform persistence theory in \cite{hale1989persistence} to obtain $\liminf_{t\rightarrow\infty} \|H_i(\cdot, t)\|_\infty+\|V_i(\cdot, t)\|_\infty \ge \epsilon$ for some $\epsilon>0$ (see Lemma \ref{lemma_persistence}). Interested readers may read the appendix on the ODE system for the idea of the proof first.

In section 4, we prove the global stability of $E_1$ for the critical case $R_0=1$. Here the main difficulty is to prove the local stability of $E_1$ as the linearized system at $E_1$ has principal eigenvalue equaling zero. In section 5, we give some concluding remarks. In particular, we summarize our results on the basic reproduction number $R_0$, which will be presented in a forthcoming paper.   We also remark that our idea is applicable to other models (e.g. \cite{lai2014repulsion, lai2016reaction,  ren2017reaction, pankavich2016mathematical}).



\section{Disease free equilibria}
The objective of this section is to define the basic reproduction number and investigate the stability of the trivial and semi-trivial steady states. The existence, uniqueness, and positivity of global classical solutions of \eqref{main}-\eqref{maini} have been shown in \cite{fitzgibbon2017outbreak}. Let $V=V_u+V_i$. Then $V(x, t)$ satisfies 
 \begin{equation} \label{v}
\left\{
\begin{array}{ll}
V_t-\triangledown\cdot \delta_2(x)\triangledown V=\beta(x) V-\mu(x) V^2,\ \ \ &x\in\Omega, t>0,\\
\frac{\partial}{\partial n}V=0,\ \ \ &x\in\partial\Omega, t>0,\\
V(., 0)=V_0\in C(\bar \Omega; \mathbb{R}_+).
\end{array}
\right.
\end{equation}
The following result about \eqref{v} is well-known (see, e.g., \cite{cantrell2004spatial}).
\begin{lemma}\label{lemma_V}
For any nonnegative nontrivial initial data $V_0\in C(\bar \Omega)$, \eqref{v} has a unique global classic solution $V(x, t)$. Moreover, $V(x, t)>0$ for all $(x, t)\in \bar\Omega\times (0, \infty)$ and
\begin{equation}
\lim_{t\rightarrow +\infty} \|V(\cdot, t)-\hat V\|_\infty =0,
\end{equation}
where $\hat V$ is the unique positive solution of the elliptic problem
\begin{equation} 
\label{vhat}
\left\{
\begin{array}{ll}
-\triangledown\cdot \delta_2(x)\triangledown V=\beta(x) V-\mu(x) V^2,\ \ \ &x \in \Omega,\\
\frac{\partial}{\partial n}V=0,\ \ \ &x\in\partial\Omega.
\end{array}
\right.
\end{equation}
\end{lemma}
By Lemma \ref{lemma_V}, $V_u(x, t)+V_i(x, t)\rightarrow \hat V(x)$ uniformly for $x\in\bar\Omega$ as $t\rightarrow\infty$ if $V_{u0}+V_{i0}\neq 0$.

 As usual, we consider two types of equilibria for \eqref{main}-\eqref{mainb}: Disease Free Equilibrium (DFE) and Endemic Equilibrium (EE).  A nonnegative solution $(\tilde  H_i, \tilde V_u, \tilde V_i)$ of \eqref{maine} is a DFE if $\tilde H_i=\tilde V_i=0$, and otherwise it is an EE. By Lemma \ref{lemma_V}, we must have $\tilde V_u+\tilde V_i=\hat V$ or $\tilde V_u+\tilde V_i=0$. It is then not hard to show that \eqref{main}-\eqref{mainb} has two DFE: trivial equilibrium $E_0=(0, 0, 0)$ and semi-trivial equilibrium $E_1=(0, \hat V, 0)$. We denote the 
 EE by $E_2=(\hat H_i, \hat V_u, \hat V_i)$, which will be proven to be unique if exists. 
\begin{lemma}
 $E_0$  is always unstable.
\end{lemma}
\begin{proof}
We linearize \eqref{main} around $E_0$ by letting $H_i(x, t)=\varphi(x) e^{\kappa t}$, $V_u(x, t)=\phi(x) e^{\kappa t}$, and $V_i(x, t)=\psi(x) e^{\kappa t}$ and ignoring the high order terms:
\begin{equation}
\left\{
\begin{array}{lll}
\kappa\varphi=\triangledown\cdot \delta_1\triangledown \varphi-\lambda \varphi+\sigma_1 H_u\psi,\ \ \ &x\in\Omega, \\
\kappa\phi=\triangledown\cdot \delta_2\triangledown \phi+\beta(\phi+\psi),\ \ \ &x\in\Omega,\\
\kappa\psi=\triangledown\cdot \delta_2\triangledown \psi,\ \ \ &x\in\Omega,\\
\frac{\partial}{\partial n}\varphi= \frac{\partial}{\partial n}\phi= \frac{\partial}{\partial n}\psi=0,\ \ \ &x\in\partial\Omega.
\end{array} \label{eig1}
\right.
\end{equation}
Let $\phi_0>0$ be a positive eigenvector corresponding to the principal eigenvalue $\tilde\kappa$ of the following problem
\begin{equation*}
\left\{
\begin{array}{lll}
\tilde\kappa \phi=\triangledown\cdot \delta_2\triangledown \phi+\beta\phi,\ \ \ &x\in\Omega,\\
\frac{\partial}{\partial n}\phi=0,\ \ \ &x\in\partial\Omega.
\end{array}
\right.
\end{equation*}
Then $\tilde\kappa>0$ is an eigenvalue of \eqref{eig1} with a corresponding eigenvector $(0, \phi_0, 0)$. Therefore $E_0$ is linearly unstable. By the principle of linearized instability, $E_0$ is unstable.
\end{proof}

Linearizing \eqref{main} around $E_1$, we arrive at the following eigenvalue problem:
\begin{equation}
\left\{
\begin{array}{lll} \label{eig}
\kappa\varphi&=\triangledown\cdot \delta_1\triangledown \varphi-\lambda \varphi+\sigma_1 H_u\psi,\ \ \ &x\in\Omega, \\
\kappa\phi&=\triangledown\cdot \delta_2\triangledown \phi-\sigma_2\hat V\varphi+\beta(\phi+\psi)-2\mu\hat V\phi-\mu\hat V\psi,\ \ \ &x\in\Omega,\\
\kappa\psi&=\triangledown\cdot \delta_2\triangledown \psi+\sigma_2\hat V\varphi-\mu \hat V\psi,\ \ \ &x\in\Omega,\\
\frac{\partial}{\partial n}\varphi&= \frac{\partial}{\partial n}\phi= \frac{\partial}{\partial n}\psi=0,\ \ \ &x\in\partial\Omega.
\end{array}
\right.
\end{equation}
Since the second equation of \eqref{eig} is decoupled from the system, we consider the problem
\begin{equation}
\left\{
\begin{array}{lll}  \label{eigg}
\kappa\varphi&=\triangledown\cdot \delta_1\triangledown \varphi-\lambda\varphi+\sigma_1 H_u\psi,\ \ \ &x\in\Omega, \\
\kappa\psi&=\triangledown\cdot \delta_2\triangledown \psi+\sigma_2\hat V\varphi-\mu \hat V\psi,\ \ \ &x\in\Omega,\\
\frac{\partial}{\partial n}\varphi&=\frac{\partial}{\partial n}\psi=0,\ \ \ &x\in\partial\Omega.
\end{array}
\right.
\end{equation}
Problem \eqref{eigg} is cooperative, so it has a principal eigenvalue $\kappa_0$ associated with a positive eigenvector $(\varphi_0, \psi_0)$ (e.g. see \cite{lam2016asymptotic}).

For $\delta\in C^1(\bar\Omega)$ being strictly positive on $\bar\Omega$ and $f\in C(\bar\Omega)$, let $\kappa_1(\delta, f)$ be the principal eigenvalue of 
\begin{equation} \label{EVP}
\left\{
\begin{array}{ll}
\kappa\phi=\triangledown\cdot \delta(x)\triangledown \phi+f \phi,\ \ \ &x\in\Omega,\\
\frac{\partial}{\partial n}\phi=0,\ \ \ &x\in\partial\Omega.
\end{array}
\right.
\end{equation}
It is well known that  $\kappa_1(\delta, f)$ is the only eigenvalue associated with a positive eigenvector, and it is monotone in the sense that if $f_1\ge(\neq) f_2$ then $\kappa_1(\delta, f_1)>\kappa_2(\delta, f_2)$.

\begin{lemma}
$E_1$ is locally asymptotically stable if $\kappa_0<0$ and unstable if $\kappa_0>0$.
\end{lemma}
\begin{proof}
Noticing that $\hat V$ is a positive solution of \eqref{vhat}, we have $\kappa_1(\delta_2, \beta-\mu \hat V)=0$. Therefore, $\kappa_1(\delta_2, -\sigma_2\hat V+\beta-2\mu\hat V)<0$.

Suppose $\kappa_0<0$. Let $\kappa$ be an eigenvalue of \eqref{eig}. Then $\kappa$ is an eigenvalue of either \eqref{eigg} or
the following eigenvalue problem:
 \begin{equation*}
 \left\{
\begin{array}{lll}
\kappa\phi&=\triangledown\cdot \delta_2\triangledown \phi-\sigma_2\hat V\phi+\beta\phi-2\mu\hat V\phi,\ \ \ &x\in\Omega,\\
\frac{\partial}{\partial n}\phi&=0,\ \ \ &x\in\partial\Omega.
\end{array}
\right.
\end{equation*}
Since $\kappa_0<0$ and $\kappa_1(\delta_2, -\sigma_2\hat V+\beta-2\mu\hat V)<0$,  we have $\mathbb{R} \kappa<0$. Since $\kappa$ is arbitrary,  $E_1$ is linearly stable. By the principle of linearized stability, $E_1$ is locally asymptotically stable.

Suppose $\kappa_0>0$. Let $(\varphi_0, \psi_0)$ be a positive eigenvector associated with $\kappa_0$. By  $\kappa_1(\delta_2, -\sigma_2\hat V+\beta-2\mu\hat V)<0$ and the Fredholm alternative, the following problem has a unique solution $\phi_0$:
 \begin{equation*}
 \left\{
\begin{array}{lll}
\kappa_0 \phi&=\triangledown\cdot \delta_2\triangledown \phi-\sigma_2\hat V\varphi+\beta(\phi+\psi_0)-2\mu\hat V\phi-\mu\hat V\psi_0,\ \ \ &x\in\Omega,\\
\frac{\partial}{\partial n}\phi&=0,\ \ \ &x\in\partial\Omega.
\end{array}
\right.
\end{equation*}
Hence \eqref{eig} has an eigenvector $(\varphi_0, \phi_0, \psi_0)$ corresponding to eigenvalue $\kappa_0>0$. So $E_1$ is linearly unstable. By the principle of linearized instability, $E_1$ is unstable.
\end{proof}

We adopt the approach of \cite{thieme2009spectral, wang2012basic} to define the basic reproduction number of \eqref{main}.
Let $B: [C(\bar\Omega)]^2 \rightarrow [C(\bar\Omega)]^2$ be the operator such that
$$D(B):= \left\{(\varphi, \psi)\in C^2(\bar\Omega; \mathbb{R}^2): \frac{\partial }{\partial n}\varphi=\frac{\partial}{\partial n}\psi=0 \ \text{ on } \partial\Omega \right \}$$
and
\begin{equation*}
B(\varphi, \psi)=
\begin{pmatrix}
\triangledown\cdot\delta_1\triangledown\varphi\\
\triangledown\cdot\delta_2\triangledown\psi
\end{pmatrix}+
 \begin{pmatrix}
-\lambda & \sigma_1 H_u \\
0 & -\mu\hat V
\end{pmatrix}
\begin{pmatrix}
\varphi\\
\psi
\end{pmatrix}, \ \ (\varphi, \psi)\in D(B).
\end{equation*}
Define
\begin{equation*}
C=\begin{pmatrix}
0 & 0 \\
\sigma_2\hat V & 0
\end{pmatrix}.
\end{equation*}
Let $A=B+C$. Then $A$ and $B$ are resolvent positive (see \cite{thieme2009spectral} for the definition), and $A$ is a positive perturbation of $B$. It is easy to check that the spectral bound of $B$ is negative, i.e. $S(B)<0$. By  \cite[Theorem 3.5]{thieme2009spectral}, $\kappa_0=S(A)$ has the same sign with $r(-CB^{-1})-1$, where $r(-CB^{-1})$ is the spectral radius of $-CB^{-1}$. Then we define the {\em basic reproduction number} $R_0$ by
$$
 R_0=r(-CB^{-1}).
$$
We immediately have the following result:
\begin{lemma}\label{lemma_R}
$R_0-1$ and $\kappa_0$ have the same sign.  Moreover, $E_1$ is locally asymptotically stable if $R_0<1$ and unstable if $R_0>1$.
\end{lemma}

We then consider the global dynamics of the model when $R_0<1$.
\begin{theorem}\label{theorem_Rless}
If $\mathcal{R}_0<1$, then $E_1$ is globally asmyptototically stable, i.e. $E_1$ is locally  stable and, for any  initial data $(H_{i0}, V_{u0}, V_{i0})\in C(\bar\Omega; \mathbb{R}_+^3)$ with $V_{u0}+V_{i0}\neq 0$, we have
\begin{equation}
\lim_{t\rightarrow\infty}\|(H_i(\cdot, t), V_u(\cdot, t), V_i(\cdot, t))-E_1\|_\infty=0.
\end{equation}
\end{theorem}
\begin{proof}
By Lemma \ref{lemma_R}, $E_1$ is locally asymptotically stable and $\kappa_0<0$. Then we can choose $\epsilon>0$ small such that the following eigenvalue problem
\begin{equation*}
\left\{
 \begin{array}{lll} \label{eigg2}
\kappa\varphi&=\triangledown\cdot \delta_1\triangledown \varphi-\lambda \varphi+\sigma_1 H_u\psi,\ \ \ &x\in\Omega, \\
\kappa\psi&=\triangledown\cdot \delta_2\triangledown \psi+\sigma_2(\hat V+\epsilon)\varphi-\mu (\hat V-\epsilon)\psi,\ \ \ &x\in\Omega,\\
\frac{\partial}{\partial n}\varphi&=\frac{\partial}{\partial n}\psi=0,\ \ \ &x\in\partial\Omega,
\end{array}
\right.
\end{equation*}
has a principal eigenvalue $\kappa_{\epsilon}<0$ with a corresponding positive eigenvector $(\varphi_\epsilon, \psi_\epsilon)$. By \eqref{main} and Lemma \ref{lemma_V}, we know that $V_u(x, t)+V_i(x, t)\rightarrow\hat V(x)$ uniformly on $\bar\Omega$  as $t\rightarrow\infty$ . Hence there exists $t_0>0$ such that $\hat V(x)-\epsilon<V_u(x, t)+V_i(x, t)<\hat V(x)+\epsilon$ for $x\in\bar\Omega$ and $t>t_0$. It then follows that
\begin{equation*}
\left\{
 \begin{array}{lll}
\frac{\partial}{\partial t}H_i-\triangledown\cdot \delta_1\triangledown H_i&=-\lambda H_i+\sigma_1 H_u(x)V_i,\ \ \ &x\in\Omega, t>t_0,\\
\frac{\partial}{\partial t}V_i-\triangledown\cdot \delta_2\triangledown V_i&\le\sigma_2(\hat V+\epsilon)H_i-\mu(\hat V-\epsilon)V_i,\ \ \ &x\in\Omega, t>t_0.
\end{array}
\right.
\end{equation*}
So $(H_i, V_i)$ is a lower solution of the following problem
\begin{equation}
\left\{
\begin{array}{lll} \label{comp}
&\frac{\partial}{\partial t}\hat H_i-\triangledown\cdot \delta_1\triangledown \hat H_i=-\lambda \hat H_i+\sigma_1 H_u\hat V_i,\ \ \ &x\in\Omega, t>t_0,\\
&\frac{\partial}{\partial t}\hat V_i-\triangledown\cdot \delta_2\triangledown \hat V_i =\sigma_2(\hat V+\epsilon)\hat H_i-\mu(\hat V-\epsilon)\hat V_i,\ \ \ &x\in\Omega, t>t_0,\\
&\frac{\partial}{\partial n}\hat H_i=\frac{\partial}{\partial n}\hat V_i=0, \ \ \ &x\in\partial\Omega, t>t_0,\\
&\hat H_i(x, t_0)=M\varphi_\epsilon(x),\ \ \ \ \hat V_i(x, t_0)=M\psi_\epsilon(x), \ \ \ &x\in\Omega,
\end{array}
\right.
\end{equation}
where $M$ is large such that $H_i(x, t_0)\le\hat H_i(x, t_0)$ and $V_i(x, t_0)\le\hat V_i(x, t_0)$. By the comparison principle for cooperative systems (e.g. \cite{smith1995monotone}), $H_i(x, t)\le \hat H_i(x, t)$ and  $V_i(x, t)\le \hat V_i(x, t)$ for all $x\in\bar\Omega$ and $t\ge t_0$. It is easy to check that the unique solution of the linear problem \eqref{comp} is $(\hat H_i(x, t), \hat V_i(x, t))=(M\varphi(x) e^{\kappa_\epsilon t}, M\psi(x) e^{\kappa_\epsilon t})$. Since $\kappa_\epsilon<0$, we have $\hat H_i(x, t)\rightarrow 0$ and $\hat V_i(x, t)\rightarrow 0$ uniformly for $x\in\bar\Omega$ as $t\rightarrow\infty$. Hence $H_i(x, t)\rightarrow 0$ and $V_i(x, t)\rightarrow 0$ uniformly for $x\in\bar\Omega$ as $t\rightarrow\infty$. By $V_{u0}+V_{i0}\neq 0$ and Lemma \ref{lemma_V}, $V_u(\cdot, t)+V_i(\cdot, t)\rightarrow\hat V$ in $C(\bar\Omega)$. So we have $V_u(x, t)\rightarrow\hat V(x)$  uniformly for $x\in\bar\Omega$ as $t\rightarrow\infty$.
\end{proof}

\section{Global dynamics when $R_0>1$}
The objective in this section is to prove the convergence of solutions of \eqref{main}-\eqref{maini} to the unique positive steady state when $R_0>1$.
\subsection{The limit problem}
By Lemma \ref{lemma_V}, we have $V_u(\cdot, t)+V_i(\cdot, t)\rightarrow \hat V$ in $C(\bar\Omega)$ as $t\rightarrow \infty$ if $V_{u0}+V_{i0}\neq 0$. This suggests us to study the following limit problem of \eqref{main}-\eqref{maini}:
\begin{equation}
\left\{
\begin{array}{lll} \label{main_reduced}
&\frac{\partial}{\partial t}H_i-\triangledown\cdot \delta_1\triangledown H_i=-\lambda H_i+\sigma_1 H_uV_i,\ \ \ &x\in\Omega, t>0,\\
&\frac{\partial}{\partial t}V_i-\triangledown\cdot \delta_2\triangledown V_i=\sigma_2 (\hat V-V_i)^+H_i-\mu\hat V V_i,\ \ \ &x\in\Omega, t>0,\\
&\frac{\partial}{\partial n}H_i = \frac{\partial}{\partial n}V_i=0,\ \ \ &x\in\partial\Omega, t>0,\\
&H_i(x, 0)=H_{i0}(x),\ V_i(x, 0)=V_{i0}(x),\ \ \ &x\in\Omega.
\end{array}
\right.
\end{equation}
The steady states of \eqref{main_reduced} are nonnegative solutions of the problem:
\begin{equation}
\left\{
\begin{array}{lll} \label{mainep}
-\triangledown\cdot \delta_1\triangledown H_i=-\lambda H_i+\sigma_1 H_uV_i,\ \ \ &x\in\Omega, \\
-\triangledown\cdot \delta_2\triangledown V_i=\sigma_2 (\hat V-V_i)^+H_i-\mu\hat V V_i,\ \ \ &x\in\Omega,\\
\frac{\partial}{\partial n}H_i= \frac{\partial}{\partial n}V_i=0,\ \ \ &x\in\partial\Omega.
\end{array}
\right.
\end{equation}
Clearly $(0, 0)$ is a steady state. In this section, we prove that if a positive steady state of \eqref{main_reduced} exists, it is globally stable in $\{(H_{i0}, V_{i0}):  H_{i0}+V_{i0}\neq 0\}$.

\subsubsection{Uniqueness of positive steady state}
In the following lemmas, we prove the uniqueness of the positive steady state of \eqref{main_reduced}.
 
\begin{lemma}\label{lemma_po}
If $(\hat H_i, \hat V_i)$ is a nontrivial nonnegative steady state, then $\hat H_i(x), \hat V_i(x)>0$ for all $x\in\bar\Omega$ and $\hat V_i(x_0)<\hat V(x_0)$ for some  $x_0\in\overline\Omega$.
\end{lemma}
\begin{proof}
Since $(\hat H_i, \hat V_i)$ is nontrivial, $\hat H_i\neq 0$ or $\hat V_i\neq 0$. Noticing $(\lambda-\triangledown\cdot \delta_1\triangledown)\hat H_i=\sigma_1H_u \hat V_i$, we must have $\hat H_i\neq 0$ and $\hat V_i\neq 0$. By the maximum principle, we have $\hat H_i(x), \hat V_i(x)>0$ for all $x\in\bar\Omega$. Assume to the contrary that $\hat V_i(x)\ge \hat V(x)$ for all $x\in\bar\Omega$, then 
$$
-\triangledown\cdot \delta_2\triangledown \hat V_i=\sigma_2 (\hat V-\hat V_i)^+\hat H_i-\mu\hat V \hat V_i=-\mu\hat V \hat V_i.
$$
This implies $\hat V_i=0$, which is a contradiction. 
\end{proof}

By the previous lemma, any nontrivial nonnegative steady state must be positive.  For any $C_1, C_2>0$, define 
$$
S=\{V_i\in C(\bar\Omega; \mathbb{R}_+): \ \ \|V_i\|_\infty \le C_1 \text{ and } V_i(x_0)<\hat V(x_0) \text{ for some } x_0\in\bar\Omega\},
$$ 
and $f: S\subset C(\bar\Omega) \rightarrow C(\bar\Omega)$ by
$$
f(V_i)=(C_2-\triangledown\cdot\delta_2\triangledown)^{-1}\left[ \sigma_2(\hat V-V_i)^+ (\lambda-\triangledown\cdot\delta_1\triangledown)^{-1}\sigma_1H_uV_i + (C_2-\mu\hat V)V_i  \right], \  \  V_i\in S.
$$

\begin{lemma}
If $(\hat H_i, \hat V_i)$ is a positive steady state, then there exists $C_1^*>0$ such that $\hat V_i$ is a nontrivial fixed point of $f$ for all $C_1>C_1^*$ and $C_2>0$. 
\end{lemma}
\begin{proof}
By the first equation of \eqref{mainep}, $\hat H_i=(\lambda-\triangledown\cdot\delta_1\triangledown)^{-1}\sigma_1H_u \hat V_i$. Substituting it into the second equation, we obtain
$$
-\triangledown\cdot \delta_2\triangledown \hat V_i=\sigma_2 (\hat V-V_i)^+(\lambda-\triangledown\cdot\delta_1\triangledown)^{-1}\sigma_1H_u \hat V_i-\mu\hat V \hat V_i.
$$
By Lemma \ref{lemma_po},  $V_i$ is a nontrivial fixed point of $f$ if $C_1$ is large. 
\end{proof}

\begin{lemma}
For any $C_1>0$, there exists $C_2^*>0$ such that $f$ is monotone for all $C_2>C_2^*$ in the sense that $f(V_i)\le f(\hat V_i)$ for all $V_i, \hat V_i\in S$ with $V_i\le \hat V_i$. 
\end{lemma}
\begin{proof}
It suffices to prove that $f(V_i)\le f(V_i+h)$ for any $V_i\in S$ and $0\le h \le C_1- V_i$. Define 
$$
\tilde f(V_i)= \sigma_2(\hat V-V_i)^+ (\lambda-\triangledown\cdot\delta_1\triangledown)^{-1}\sigma_1H_uV_i + (C_2-\mu\hat V)V_i. 
$$
 Then, we have
\begin{equation*}
\begin{split}
&\tilde f(V_i+h) -\tilde f(V_i)= \sigma_2( (\hat V-V_i-h)^+ - (\hat V-V_i)^+ )(\lambda-\triangledown\cdot\delta_1\triangledown)^{-1}\sigma_1H_uV_i  \\
&\hspace{3.7cm} +\sigma_2(\hat V-V_i-h)^+(\lambda-\triangledown\cdot\delta_1\triangledown)^{-1}\sigma_1H_uh+ (C_2-\mu\hat V)h\\
&\hspace{3.7cm} \ge [-\sigma_2(\lambda-\triangledown\cdot\delta_1\triangledown)^{-1}\sigma_1H_uV_i+C_2-\mu\hat V]h,
\end{split}
\end{equation*}
where we used 
$$
|(\hat V-V_i-h)^+ - (\hat V-V_i)^+| \le h.
$$ 
By the elliptic estimate,  the following set is bounded: 
$$
\{ (\lambda-\triangledown\cdot\delta_1\triangledown)^{-1}\sigma_1H_uV_i, \ \ V_i\in S\}.
$$ 
Hence, $\tilde f(V_i+h)-\tilde f(V_i)\ge 0$ if $C_2$ is large. Therefore,  $f(V_i+h)-f(V_i)\ge 0$, and $f$ is monotone. 
\end{proof}

For any $f_1, f_2\in C(\bar\Omega)$, we say $f_1<<f_2$ if $f_1(x)<f_2(x)$ for all $x\in\bar\Omega$.
\begin{lemma}
For any $k\in (0, 1)$ and $V_i\in S$ with $V_i>> 0$, $kf(V_i)<<f(k V_i)$.
\end{lemma}
\begin{proof}
By the definition of $S$, there exists $x_0\in\bar\Omega$ such that $\hat V(x_0)> V_i(x_0)$. So $(\hat V(x_0)- V_i(x_0))^+<(\hat V(x_0)-k V_i(x_0))^+$ and $(\hat V(x)- V_i(x))^+\le (\hat V(x)-k V_i(x))^+$ for all $x\in\bar\Omega$. It then follows that $k\tilde f(V_i)(x_0)<\tilde f(kV_i)(x_0)$ and $k\tilde f(V_i)\le \tilde f(kV_i)$. The assertion now just follows from the fact that $(C_2-\triangledown\cdot\delta_2\triangledown)^{-1}$ is strongly positive (i.e. if $g\in C(\bar\Omega)$ such that $g\ge 0$ and $g(x_0)>0$ for some $x_0\in\bar\Omega$, then $(C_2-\triangledown\cdot\delta_2\triangledown)^{-1}g>>0$).
\end{proof}

\begin{lemma}\label{lemma_unique}
The positive steady state, if exists, is unique.
\end{lemma}
\begin{proof}
Suppose to the contrary that $(H_i^1, V_i^1)$ and $(H_i^2, V_i^2)$ are two distinct positive steady states. Then $V_i^1\neq V_i^2$ by the first equation of \eqref{mainep}.  Without loss of generality, we may assume $V_i^1\not\le V_i^2$. Define
$$
k=\max\{\tilde k\ge 0: \ \ \tilde k V_i^1\le V_i^2 \}.
$$
Then $k\in (0, 1)$. By the definition of $k$, $ k V_i^1\le V_i^2$ and $k V_i^1(x_0)= V_i^2(x_0)$ for some $x_0\in\bar\Omega$. We can choose $C_1$ and $C_2$ such that they are fixed points of $f$, i.e. $f(V^1_i)=V^1_i$ and $f(V^2_i)=V^2_i$. By the previous lemmas, we have
$$
kV^1_i=kf(V^1_i)<< f(k V^1_i)\le f(V^2_i)=V^2_i.
$$
Thus $kV^1_i<<V^2_i$, which contradicts $k V_i^1(x_0)= V_i^2(x_0)$.
\end{proof}

\subsubsection{Global stability of positive steady state}
Let $F_1(H_i, V_i)=-\lambda H_i+\sigma_1 H_u V_i$ and $F_2(H_i, V_i)=\sigma_2 (\hat V-V_i)^+H_i-\mu \hat V V_i$. Noticing $\partial F_1/\partial V_i\ge 0$ and $\partial F_2/\partial H_i\ge 0$,  system \eqref{main_reduced} is cooperative. Let $\tilde \Phi(t): C(\bar\Omega; \mathbb{R}^2)\rightarrow C(\bar\Omega; \mathbb{R}^2)$ be the semiflow induced by the solution of \eqref{main_reduced}, i.e. $\tilde \Phi(t)(H_{i0}, V_{i0})=(H_i(\cdot, t), V_i(\cdot, t))$ for all $t\ge 0$. Then $\tilde \Phi(t)$ is monotone (e.g. see \cite{smith1995monotone}). 

\begin{lemma}\label{lemma_positive}
For any nonnegative nontrivial initial data $(H_{i0}, V_{i0})$, the solution of \eqref{main_reduced} satisfies that $H_i(x, t)>0$ and $V_i(x, t)>0$ for all $x\in\bar\Omega$ and $t>0$.
\end{lemma}
\begin{proof}
By the comparison principle for cooperative systems, $H_i(x, t)\ge 0$ and $V_i(x, t)\ge 0$ for all $x\in\bar\Omega$ and $t\ge 0$. 
Suppose $V_{i0}\neq 0$. Noticing 
\begin{equation}\label{e11}
\frac{\partial}{\partial t}V_i-\triangledown\cdot \delta_2\triangledown V_i\ge -\mu\hat V V_i
\end{equation}
and by the comparison principle, we have $V_i(x, t)>0$ for all $x\in\bar\Omega$ and $t>0$. By
$$
\frac{\partial}{\partial t}H_i-\triangledown\cdot \delta_1\triangledown H_i > -\lambda H_i,
$$
and the comparison principle, $H_i(x, t)>0$ for all $x\in \bar\Omega$ and $t>0$.  

Suppose  $V_{i0}= 0$. Since $(H_{i0}, V_{i0})$ is nontrivial, we have $H_{i0}\neq 0$. By
$$
\frac{\partial}{\partial t}H_i-\triangledown\cdot \delta_1\triangledown H_i \ge -\lambda H_i,
$$
and the comparison principle, we have $H_i(x, t)>0$ for all $x\in \bar\Omega$ and $t>0$.  By the continuity  of $V_i(x, t)$ and $V_i(x, 0)=0$,  $(\hat V-V_i(x, t))^+>0$ for all $(x, t)\in \bar\Omega\times (0, t_0]$ for some $t_0>0$. Then by
\begin{equation*}
\frac{\partial}{\partial t}V_i-\triangledown\cdot \delta_2\triangledown V_i> -\mu\hat V V_i, \ \ \ x\in\bar\Omega, t\in (0, t_0]
\end{equation*}
and the comparison principle, we have $V_i(x, t)>0$ for all $(x, t)\in\bar\Omega\times (0, t_0]$. Finally by \eqref{e11}, we have $V_i(x, t)>0$ for all $x\in\bar\Omega$ and $t>0$.
\end{proof}

\begin{lemma}
For any nonnegative initial data $(H_{i0}, V_{i0})$, there exists $M>0$ such that the solution of \eqref{main_reduced} satisfies
$$
0\le H_i(x, t), V_i(x, t)\le M, \ \ \text{ for all } x\in\bar\Omega, t>0.
$$
\end{lemma}
\begin{proof}
 Let  $M_1=\max\{\|\hat V\|_\infty, \|V_{i0}\|_\infty\}$. By the second equation of \eqref{main_reduced} and the comparison principle, we have $V_i(x, t)\le M_1$ for all $x\in \bar\Omega$ and $t>0$. Then by the first equation of \eqref{main_reduced}, we have
$$
\frac{\partial}{\partial t}H_i-\triangledown\cdot \delta_1(x)\triangledown H_i\le -\lambda(x) H_i+\sigma_1(x) H_u(x)M_1,\ \ \ x\in\Omega, t>0.
$$
So $H_i$ is a lower solution of the problem:
\begin{equation*}
\left\{
\begin{array}{lll} \label{compH}
\frac{\partial}{\partial t}w-\triangledown\cdot \delta_1(x)\triangledown w=-\lambda(x)w+\sigma_1(x) H_u(x)M_1,\ \ \ &x\in\Omega, t>0,\\
\frac{\partial}{\partial n}w = 0,\ \ \ &x\in\partial\Omega, t>0,\\
w(x, 0)=H_{i0}(x),\ \ \ &x\in\Omega.
\end{array}
\right.
\end{equation*}
Let $M_2=\max\{\|\sigma_1\|_\infty\|H_u\|_\infty M_1/\lambda_m,  \  \|H_{i0}\|_\infty \}$, where $\lambda_m=\min\{\lambda(x): x\in\bar\Omega\}$. Then we have $0\le w(x, t)\le M$ for all $(x, t)\in \bar\Omega\times (0, \infty)$. Hence by the comparison principle,  we have $0\le H_i(x, t)\le w(x, t)<M_2$. Therefore, the claim holds for  $M=\max\{M_1, M_2\}$.
\end{proof}

\begin{lemma}\label{lemma_reduced}
If the positive steady state $(\hat H_i, \hat V_i)$ of \eqref{main_reduced} exists, it is globally asymptotically stable, i.e. it is locally  stable and, for any nonnegative nontrivial initial data $(H_{i0}, V_{i0})$,
$$
\lim_{t\rightarrow\infty} H_i(\cdot, t)=\hat H_i \ \text{ and } \ \ \lim_{t\rightarrow\infty} V_i(\cdot, t)=\hat V_i  \ \ \text{ in } C(\bar\Omega).
$$
\end{lemma}
\begin{proof}
By Lemma \ref{lemma_positive}, we have $H_i(x, t)>0$ and $V_i(x, t)>0$ for all $x\in\bar\Omega$ and $t>0$. So without loss of generality, we may assume $H_{i0}(x)>0$ and $V_{i0}(x)>0$ for all $x\in\bar\Omega$.

Suppose that $(\hat H_i, \hat V_i)$ is a positive steady state of \eqref{main_reduced}, which is unique by Lemma \ref{lemma_unique}. Let $(\underline H_i, \underline V_i)= (\epsilon\hat H_i, \epsilon\hat V_i)$ for some $\epsilon>0$. We may choose $\epsilon$ small such that the following is satisfied:
 \begin{equation}
 \left\{
\begin{array}{lll} \label{compless}
-\triangledown\cdot \delta_1(x)\triangledown\underline H_i\le -\lambda(x) \underline H_i+\sigma_1(x) H_u(x)\underline V_i,\ \ \ &x\in\Omega, \\
-\triangledown\cdot \delta_2(x)\triangledown\underline V_i\le \sigma_2(x) (\hat V-\underline V_i)^+\underline H_i-\mu(x)\hat V \underline V_i,\ \ \ &x\in\Omega,\\
\frac{\partial}{\partial n}\underline H_i= \frac{\partial}{\partial n}\underline V_i=0,\ \ \ &x\in\partial\Omega,\\
\underline H_i(x)\le H_{i0}(x),\ \underline V_i(x)\le V_{i0}(x),\ \ \ &x\in\Omega.
\end{array}
\right.
\end{equation}
Hence by \cite[Corollary 7.3.6]{smith1995monotone}, $\tilde \Phi(t)(\underline H_i, \underline V_i)$ is monotone increasing in $t$ and converges to a positive steady state of \eqref{main_reduced}. Since $(\hat H_i, \hat V_i)$ is the unique positive steady state of \eqref{main_reduced}, we must have $\tilde \Phi(t)(\underline H_i, \underline V_i)\rightarrow (\hat H_i, \hat V_i)$ in $C(\bar\Omega)$ as $t\rightarrow \infty$.

Similarly, we may  define $(\overline H_i, \overline V_i)= (k\hat H_i, k\hat V_i)$ with $k$ large  such that  \eqref{compless} is satisfied with inverse inequalities, and then  $\tilde \Phi(t)(\overline H_i, \overline V_i)\rightarrow (\hat H_i, \hat V_i)$ in $C(\bar\Omega)$ as $t\rightarrow \infty$. Since $(\underline H_i, \underline V_i)\le (H_{i0}, V_{i0}) \le (\overline H_i, \overline V_i)$ and $\tilde\Phi(t)$ is monotone, we have  $\tilde\Phi(t)(\underline H_i, \underline V_i)\le \tilde\Phi(t)(H_{i0}, V_{i0})\le \tilde\Phi(t)(\overline H_i, \overline V_i)$ for all $t\ge 0$. Therefore, $\tilde\Phi(t)(H_{i0}, V_{i0})\rightarrow (\hat H_i, \hat V_i)$ in $C(\bar\Omega)$ as $t\rightarrow \infty$, and $(\hat H_i, \hat V_i)$ is locally stable. This proves the lemma.
\end{proof}


\subsection{Global stability of $E_2$}
 In this section, we prove the convergence of solutions of \eqref{main}-\eqref{maini} to the unique positive steady state $E_2$ when $R_0>1$. We begin by proving the ultimate boundedness of the solutions. 
\begin{lemma}\label{lemma_bound}
There exists $M>0$, independent of initial data, such that any solution $(H_i, V_u, V_i)$ of\eqref{main}-\eqref{maini} satisfies that
\begin{equation*}
0\le H_i(x, t), V_u(x, t), V_i(x, t)\le M, \ \ \ x\in\bar\Omega, t\ge t_0,
\end{equation*}
where $t_0$ is dependent on initial data.
\end{lemma}
\begin{proof}
By Lemma \ref{lemma_V}, we have $V_u(x, t)+V_i(x, t)\rightarrow\hat V(x)$ uniformly on $\bar\Omega$ as $t\rightarrow\infty$ if $V_{u0}+V_{i0}\neq 0$. Hence there exists $t_1>0$ depending on initial data such that $V_u(x, t)+V_i(x, t)\le\|\hat V\|_\infty+1$ for $t>t_1$ and $x\in\bar\Omega$. By the first equation of \eqref{main} and the comparison principle, we have $H_i\le \hat H_i$ on $\bar\Omega\times [t_1, \infty)$, where  $\hat H_i$ is the solution of the problem
 \begin{equation*}
 \left\{
 \begin{array}{lll}
\frac{\partial}{\partial t}\hat H_i-\triangledown\cdot \delta_1(x)\triangledown \hat H_i=-\lambda(x) \hat H_i+\sigma_1(x) H_u(x)(\|\hat V\|_\infty+1),\ \ \ &x\in\Omega, t>t_1,\\
\frac{\partial}{\partial n}\hat H_i=0, \ \ \ &x\in\partial\Omega, t>t_1,\\
\hat H_i(x, t_1)=H_i(x, t_1), \ \ \ &x\in\Omega,
\end{array}
\right.
\end{equation*}
We know that $\hat H_i(x, t)\rightarrow \hat H_{i}^*(x)$  uniformly on $\bar\Omega$ as $t\rightarrow \infty$, where $\hat H_{i}^*$ is the unique solution of the problem
\begin{equation*}
\left\{
 \begin{array}{lll}
-\triangledown\cdot \delta_1(x)\triangledown \hat H_i=-\lambda(x) \hat H_i+\sigma_1(x) H_u(x)(\|\hat V\|_\infty+1),\ \ \ &x\in\Omega, \\
\frac{\partial}{\partial n}\hat H_i=0, \ \ \ &x\in\partial\Omega.
\end{array}
\right.
\end{equation*}
Therefore there exists $t_0>t_1$ such that $H_i(x, t)\le\hat H_i(x, t)<\|\hat H_{i}^*\|_\infty+1$ for all $x\in\bar\Omega$ and $t\ge t_0$. Therefore, the claim holds with $M=\max\{\|\hat V\|_\infty+1, \|\hat H_{i}^*\|_\infty+1\}$.
\end{proof}

Let $(X, d)$ be a complete metric space and $\Phi(t): X\rightarrow X$ be a continuous semiflow. The distance from a subset $B$ of $X$ to a subset $A$ of $X$ is defined as $d(B, A):=\sup_{y\in B}\inf_{x\in A}d(y, x)$. Suppose that $X=\bar X_0$, where  $X_0$ is an open subset of $X$. Then $X=X_0\cup \partial X_0$ with the boundary $\partial X_0=X-X_0$ being closed in $X$. The semiflow $\Phi(t)$ is said to be uniformly persistent with respect to $(X_0, \partial X_0)$ if there is an $\epsilon>0$ such that $\liminf_{t\rightarrow \infty} d(T(t)x, \partial X_0)\ge \epsilon$ for all $x\in X_0$.

In the following of this section, let $X=C(\bar\Omega; \mathbb{R}_+^3)$ with the metric induced by the supremum norm $\|\cdot\|_\infty$. Define
$$
\partial X_0:=\lbrace (H_i,V_u,V_i) \in X: H_i+V_i=0 \text{ or }  V_u+V_i=0  \rbrace
$$
and
$$
X_0:=\lbrace (H_i,V_u,V_i) \in X: H_i+V_i>0 \text{ and }  V_u+V_i>0  \rbrace.
$$
Then $X=X_0\cup \partial X_0$, $X_0$ is relatively open with $\bar X_0=X$, and $\partial X_0$ is relatively closed in $X$. Let $w(x, t)=(H_i(x, t), V_u(x, t), V_i(x, t))$ be the solution of \eqref{main}-\eqref{maini} with initial data $w_0=(H_{i0}, V_{u0}, V_{i0})\in X$. Let $\Phi(t): X\rightarrow X$ be the semiflow induced by the solution of \eqref{main}-\eqref{maini}, i.e. $\Phi(t)w_0=w(\cdot, t)$ for $t\ge 0$. Then $\Phi(t)$ is point dissipative by Lemma \ref{lemma_bound} (see, e.g., \cite{hale1989persistence} for the definition). Moreover, $\Phi(t)$ is compact for any $t>0$, since \eqref{main}-\eqref{maini} is a standard reaction-diffusion system.

We prove the following persistence result when $R_0>1$, which is necessary for proving the convergence of solutions to the positive steady state.
\begin{lemma}\label{lemma_persistence}
If $R_0>1$, then \eqref{main}-\eqref{maini} is uniformly persistent in the sense that there exists $\epsilon>0$ such that, for any initial data $(H_{i0}, V_{u0}, V_{i0})\in X_0$,
\begin{equation}\label{strong1}
\liminf_{t\rightarrow\infty} \inf_{w\in \partial X_0}\|(H_i(\cdot, t), V_u(\cdot, t), V_i(\cdot, t))- w\|_\infty\ge \epsilon.
\end{equation}
Moreover, \eqref{main}-\eqref{maini} has at least one EE.
\end{lemma}
\begin{proof}
We prove this result in several steps.

Step 1.  $X_0$ is invariant under $\Phi(t)$.

Let $w_0=(H_{i0}, V_{u0}, V_{i0})\in X_0$. Then  $H_{i0}+V_{i0}>$ and $V_{u0}+V_{i0}> 0$. Suppose $V_{i0}=0$. Then  $H_{i0}\neq 0$ and $V_{u0}\neq 0$.
By the first equation of \eqref{main}, we have
$$\frac{\partial}{\partial t}H_i-\triangledown\cdot \delta_1\triangledown H_i\ge -\lambda H_i.$$
Then  by $H_{i0}\neq 0$ and the maximum principle, we have $H_i(x, t)>0$ for $x\in\bar\Omega$ and $t>0$.
 By the second equation of \eqref{main}, we have
$$\frac{\partial}{\partial t}V_u-\triangledown\cdot \delta_2\triangledown V_u\ge V_u(-\sigma_2H_i+\beta-\mu(V_u+V_i)).$$
Then by $V_{u0}\neq 0$ and the maximum principle, we have $V_u(x, t)>0$ for $x\in\bar\Omega$ and $t>0$. Noticing the third equation of \eqref{main}, we have
 $$\frac{\partial}{\partial t}V_i-\triangledown\cdot \delta_2\triangledown V_i> -\mu(V_u+V_i) V_i, \ \ \ x\in\Omega, t>0.$$
Then by the maximum principle, we have $V_i(x, t)>0$ for $x\in\bar\Omega$ and $t>0$.\\
Suppose $V_{i0}\neq 0$. Noticing
 $$
 \frac{\partial}{\partial t}V_i-\triangledown\cdot \delta_2\triangledown V_i\ge  -\mu(V_u+V_i) V_i
 $$
 and by the maximum principle, we have $V_i(x, t)>0$ for all $x\in\bar\Omega$ and $t>0$. By the first equation of \eqref{main}, we have
$$\frac{\partial}{\partial t}H_i-\triangledown\cdot \delta_1\triangledown H_i> -\lambda H_i, \ \ \  x\in\Omega, t>0,$$
which implies that  $H_i(x, t)>0$ for all $x\in\bar\Omega$ and $t>0$. By the second equation of \eqref{main}, we have
$$
\frac{\partial}{\partial t}V_u-\triangledown\cdot \delta_2\triangledown V_u> V_u(-\sigma_2H_i+\beta-\mu(V_u+V_i)), \ \ \  x\in\Omega, t>0.
$$
Therefore, we have $\Phi(t)w_0\in X_0$ for all $t>0$. Hence $X_0$ is invariant under $\Phi(t)$.

Step 2. $\partial X_0$ is invariant under $\Phi(t)$. For any $w_0\in \partial X_0$, the $\omega$-limit set $\omega(w_0)$ is either $\{E_0\}$ or $\{E_1\}$.

Suppose $w_0=(H_{i0}, V_{u0}, V_{i0})\in \partial X_0$. Then, $H_{i0}+V_{i0}= 0$ or $V_{u0}+V_{i0}= 0$. If $H_{i0}+V_{i0}= 0$ and  $V_{u0}\neq 0$, then we have $H_i(\cdot, t)=V_i(\cdot, t)= 0$ for all $t\ge 0$ by the first and third equations of \eqref{main}. Then the second equation of \eqref{main} is
$$
\frac{\partial}{\partial t}V_u-\triangledown\cdot \delta_2\triangledown V_u= V_u(\beta-\mu V_u).
$$
Hence by Lemma \ref{lemma_V}, we have $V_u(x, t)>0$ for $x\in\bar\Omega$ and $t>0$, and $V_u(\cdot, t)\rightarrow \hat V$ uniformly on $\bar\Omega$ as $t\rightarrow \infty$. So $\Phi(t)w_0\in\partial X_0$ with $\omega(w_0)=\{E_1\}$.  

 If $V_{u0}+V_{i0}= 0$, then by the second and third equations of \eqref{main}, we have $V_u(\cdot, t)=V_i(\cdot, t)= 0$ for all $t\ge 0$. Then the first equation of \eqref{main} is
$$
\frac{\partial}{\partial t}H_i-\triangledown\cdot \delta_1\triangledown H_i= -\lambda H_i, 
$$
which implies that $H_i(x, t)\rightarrow 0$ uniformly on $\bar\Omega$ as $t\rightarrow 0$. Therefore, we have $\Phi(t)w_0\in\partial X_0$ with $\omega(w_0)=\{E_0\}$.

By Step 2, the semiflow $\Phi_\partial (t):=\Phi(t)|_{\partial X_0}$, the restriction of $\Phi(t)$ on $\partial X_0$,  admits a compact global attractor $A_\partial$.  Moreover, it is clear that
$$\tilde A_\partial:=\cup_{w_0\in A_\partial} \omega(w_0)=\{E_0, E_1\}.$$

 Step 3. $\tilde A_\partial$ has an acyclic covering $M=\{E_0\}\cup \{E_1\}$.

It suffices to show that $\{E_1\}\not\rightarrow \{E_0\}$, i.e. $W^u(E_1)\cap W^s(E_0)=\varnothing$. Suppose to the contrary that there exists $w_0=(H_{i0}, V_{u0}, V_{i0})\in W^u(E_1)\cap W^s(E_0)$. Let $(H_i(\cdot, t), V_u(\cdot, t), V_i(\cdot, t))$ be a complete orbit through $w_0$. By $w_0\in W^s(E_0)$ and Lemma \ref{lemma_V},  we have $V_{u0}=V_{i0}=0$, and hence $V_u(\cdot, t)=V_i(\cdot, t)=0$ for all $t\in (-\infty, \infty)$. Therefore $V_u(\cdot, t)\not\rightarrow \hat V$ as $t\rightarrow -\infty$, contradicting $w_0\in W^u(E_1)$.  Therefore $M=\{E_0\}\cup \{E_1\}$ is an acyclic covering of $\tilde A_\partial$. 

Step 4. $W^s(E_0)\cap X_0=\varnothing$ and $W^s(E_1)\cap X_0=\varnothing$.

We will actually show:
\begin{equation}\label{E0s}
W^s(E_0)=\{(H_{i0}, V_{u0}, V_{i0})\in\partial X_0: \ V_{u0}=V_{i0}=0\}
\end{equation}
and
$$
W^s(E_1)=\{(H_{i0}, V_{u0}, V_{i0})\in\partial X_0: \  H_{i0}=V_{i0}=0 \ \ \text{and}\ \ V_{u0}\not\equiv 0\}.
$$

By Step 2, it suffices to show that there exists $\epsilon>0$ such that, for any  initial data $(H_{i0}, V_{u0}, V_{i0})\in X_0$, we have
  \begin{equation}\label{weak1}
\limsup_{t\rightarrow\infty}\|(H_i(\cdot, t), V_u(\cdot, t), V_i(\cdot, t))-E_0\|_\infty\ge \epsilon
\end{equation}
and
 \begin{equation}\label{weak2}
\limsup_{t\rightarrow\infty}\|(H_i(\cdot, t), V_u(\cdot, t), V_i(\cdot, t))-E_1\|_\infty\ge \epsilon.
\end{equation}

We first prove \eqref{weak2}. By Lemma \ref{lemma_R} and $R_0>1$, we have $\kappa_0>0$. Hence there exists $\epsilon_0>0$ such that the following problem has a principal eigenvalue $\kappa_{\epsilon_0}>0$ corresponding to a positive eigenvector $(\phi_{\epsilon_0}, \psi_{\epsilon_0})$
\begin{equation*}
\left\{
 \begin{array}{lll}
\kappa\varphi=\triangledown\cdot \delta_1\triangledown \varphi-\lambda \varphi+\sigma_1 H_u\psi,\ \ \ &x\in\Omega, \\
\kappa\psi=\triangledown\cdot \delta_2\triangledown \psi+\sigma_2(\hat V-\epsilon_0)\varphi-\mu (\hat V+2\epsilon_0)\psi,\ \ \ &x\in\Omega,\\
\frac{\partial}{\partial n}\varphi=\frac{\partial}{\partial n}\psi=0,\ \ \ &x\in\partial\Omega.
\end{array}
\right.
\end{equation*}

Assume to the contrary that \eqref{weak2} does not hold. Then there exists some $w_0=(H_{i0}, V_{u0}, V_{i0})\in X_0$ such that the corresponding solution satisfies
\begin{equation*}
\limsup_{t\rightarrow\infty}\|(H_i(\cdot, t), V_u(\cdot, t), V_i(\cdot, t))-E_1\|_\infty< \epsilon_0.
\end{equation*}
Hence there exists $t_0>0$ such that $\hat V-\epsilon_0<V_u(\cdot, t)<\hat V+\epsilon_0$ and $V_i(\cdot, t)<\epsilon_0$ for all $t\ge t_0$. It then follows from the second and third equation of \eqref{main} that
\begin{equation*}
\left\{
\begin{array}{lll}
&\frac{\partial}{\partial t}H_i-\triangledown\cdot \delta_1\triangledown H_i=-\lambda H_i+\sigma_1 H_uV_i,\ \ \ &x\in\Omega, t\ge t_0,\\
&\frac{\partial}{\partial t}V_i-\triangledown\cdot \delta_2\triangledown V_i\ge \sigma_2(\hat V-\epsilon_0)H_i-\mu(\hat V+2\epsilon_0)V_i,\ \ \ &x\in\Omega, t\ge t_0.
\end{array}
\right.
\end{equation*}
In Step 1, we have shown that $H_i(x, t), V_i(x, t)>0$ for all $x\in\bar\Omega$ and $t>0$. Thus we can choose $m>0$ small such that $H_i(\cdot, t_0)\ge m\phi_{\epsilon 0}$ and $V_i(\cdot, t_0)\ge m\psi_{\epsilon 0}$. Hence $(H_i, V_i)$ is an upper solution of the problem
\begin{equation*}
\left\{
\begin{array}{lll}
&\frac{\partial}{\partial t}\bar H_i-\triangledown\cdot \delta_1\triangledown \bar H_i=-\lambda \bar H_i+\sigma_1 H_u\bar V_i,\ \ \ &x\in\Omega, t\ge t_0,\\
&\frac{\partial}{\partial t}\bar V_i-\triangledown\cdot \delta_2\triangledown \bar V_i= \sigma_2 (\hat V-\epsilon_0)\bar H_i-\mu(\hat V+2\epsilon_0)\bar V_i,\ \ \ &x\in\Omega, t\ge t_0,\\
&\frac{\partial }{\partial n}\bar H_i=\frac{\partial}{\partial n}\bar V_i=0,\ \ \ &x\in\partial\Omega, t\ge t_0,\\
&\bar H_i(\cdot, t_0)=m\phi_{\epsilon 0},\  \  \bar V_i(\cdot, t_0)=m\psi_{\epsilon 0}.
\end{array}
\right.
\end{equation*}
We observe that the solution of this problem is $(\bar H_i, \bar V_i)=m e^{\kappa_{\epsilon 0}t}(\phi_{\epsilon_0}, \psi_{\epsilon_0})$. By the comparison principle of the cooperative systems, we have $H_i(\cdot, t)\ge \bar H_i(\cdot, t)$ and $V_i(\cdot, t)\ge \bar V_i(\cdot, t)$ for $t\ge t_0$. Since $\kappa_{\epsilon 0}>0$, we have $H_i(\cdot, t)\rightarrow \infty$ and $V_i(\cdot, t)\rightarrow\infty$ as $t\rightarrow\infty$, which contradicts the boundedness of the solution.  This proves \eqref{weak2}.

We then prove \eqref{weak1}. Suppose to the contray that \eqref{weak1} does not hold. Then for given small $\epsilon_1>0$, there exists initial data $(H_{i0}, V_{u0}, V_{i0})\in X_0$ such that
  \begin{equation*}
\limsup_{t\rightarrow\infty}\|(H_i(\cdot, t), V_u(\cdot, t), V_i(\cdot, t))-E_0\|_\infty< \epsilon_1.
\end{equation*}
Hence there exists $t_1>0$ such that $V_u(\cdot, t)<\epsilon_1$ and $V_i(\cdot, t)<\epsilon_1$ for all $t\ge t_1$. However by Lemma \ref{lemma_V}, we know that $V_u(\cdot, t)+V_i(\cdot, t)\rightarrow \hat V$ uniformly on $\bar\Omega$ as $t\rightarrow\infty$, which is a contradiction as $\epsilon_1$ is small.

Finally by Steps 1-4 and \cite[Theorem 4.1]{hale1989persistence}, there exists $\epsilon>0$ such that \eqref{strong1}  holds. Moreover by \cite[Theorem1.3.7]{zhao2013dynamical}, \eqref{main}-\eqref{maini} has an EE.
\end{proof}

Combing Lemmas \ref{lemma_reduced} and \ref{lemma_persistence}, we can prove the main result in this section. 
\begin{theorem}\label{theorem_Rbigger}
If $R_0>1$, then for any initial data $(H_{i0}, V_{u0}, V_{i0})\in X_0$, the solution $(H_i, V_u, V_i)$ of \eqref{main}-\eqref{maini} satisfies that 
\begin{equation*}
\lim_{t\rightarrow\infty} (H_i(x, t), V_u(x, t), V_i(x, t))=(\hat H_i, \hat V_u, \hat V_i) \ \ \text{uniformly on } \bar\Omega,
\end{equation*}
where $E_2=(\hat H_i, \hat V_u, \hat V_i)$ is the unique EE of \eqref{main}.
\end{theorem}
\begin{proof}
By Lemma \ref{lemma_persistence}, there exists an EE, $E_2:=(\hat H_i, \hat V_u, \hat V_i)$, of \eqref{main}-\eqref{maini} when $R_0>1$. By Lemma \ref{lemma_V},  $\hat V_u+\hat V_i=\hat V$. So $(\hat H_i, \hat V_i)$ is a positive solution of \eqref{mainep}, which is unique by Lemma \ref{lemma_unique}. Hence, $E_2$ is the unique EE of \eqref{main}-\eqref{maini}.

 Let $(H_{i0}, V_{u0}, V_{i0})\in X_0$. Then $V_{u0}+V_{i0}\neq 0$ and $H_{i0}+V_{i0}\neq 0$.  By Lemma \ref{lemma_V}, we have $V_u(\cdot, t)+V_i(\cdot, t)\rightarrow \hat V$ in $C(\bar\Omega)$ as $t\rightarrow\infty$.  By Lemma \ref{lemma_persistence}, there exists $\epsilon>0$ such that 
 \begin{equation}\label{per}
 \liminf_{t\rightarrow\infty} \|H_i(\cdot, t)\|_\infty+\|V_i(\cdot, t)\|_\infty \ge \epsilon.
 \end{equation}
 
 We focus on the first and third equations of \eqref{main} and rewrite them as: 
 \begin{equation}
\left\{
\begin{array}{lll} \label{main_reduced1}
\frac{\partial}{\partial t}H_i-\triangledown\cdot \delta_1\triangledown H_i=-\lambda H_i+\sigma_1 H_uV_i,\ \ \ &x\in\Omega, t>0,\\
\frac{\partial}{\partial t}V_i-\triangledown\cdot \delta_2\triangledown V_i=\sigma_2 (\hat V-V_i)^+H_i-\mu\hat V V_i+F(x, t),\ \ \ &x\in\Omega, t>0,\\
\frac{\partial}{\partial n}H_i = \frac{\partial}{\partial n}V_i=0,\ \ \ &x\in\partial\Omega, t>0,\\
H_i(x, 0)=H_{i0}(x),\ V_i(x, 0)=V_{i0}(x),\ \ \ &x\in\Omega,
\end{array}
\right.
\end{equation}
 where 
 $$
 F(x, t)=\sigma_2(V_u(\cdot, t)-(\hat V-V_i(\cdot, t))^+)H_i-\mu(V_u(\cdot, t)+V_i(\cdot, t)-\hat V).
 $$
Noticing 
$$
|V_u(\cdot, t)-(\hat V-V_i(\cdot, t))^+|\le |V_u(\cdot, t)+V_i(\cdot, t)-\hat V|, 
$$
we have  $F(x, t)\rightarrow 0$ uniformly on $\bar\Omega$ as $t\rightarrow\infty$. Then by  \cite[Proposition 1.1]{mischaikow1995asymptotically}, \eqref{main_reduced1} is asymptotically autonomous with limit system \eqref{main_reduced}. By \eqref{per}, the $\omega-$limit set of \eqref{main_reduced1} is contained in $W:=\{(H_i, V_i)\in C(\bar\Omega; \mathbb{R}_+^2):  \ H_i+V_i\neq 0\}$. By Lemma \ref{lemma_reduced}, $W$ is the stable set (or basin of attraction) of the equilibrium $(\hat H_i, \hat V_i)$ of \eqref{main_reduced}. Hence by the theory of asymptotically autonomous semiflows (originally due to Markus. See \cite[Theorem 4.1]{thieme1992convergence} for the generalization to asymptotically autonomous semiflows), we have $(H_i(\cdot, t), V_i(\cdot, t))\rightarrow (\hat H_i, \hat V_i)$ in $C(\bar\Omega; \mathbb{R}^2)$ as $t\rightarrow\infty$. Moreover, by $V_u(\cdot, t)+V_i(\cdot, t)\rightarrow \hat V$ and $\hat V_i+\hat V_u=\hat V$, we have $V_u(\cdot, t)\rightarrow \hat V_u$ in $C(\bar\Omega)$ as $t\rightarrow\infty$. This completes the proof.
\end{proof}

\section{Global stability when $R_0=1$}
In this section, we prove the global stability of $E_1$ for the critical case $R_0=1$. The following result is well known. Since we can not locate a reference and for the convenience of readers, we attach a proof. 
\begin{lemma}\label{lemma_exp}
The positive steady state $\hat V$ of \eqref{v} is exponentially asymptotically stable.
\end{lemma}
\begin{proof}
It is easy to see that $\hat V$ is locally asymptotically stable. To see this, linearizing \eqref{v} around $\hat V$, we obtain
\begin{equation}
\left\{
\begin{array}{ll}
\kappa \phi=\triangledown\cdot \delta_2\triangledown \phi+\beta \phi-2\mu \hat V\phi,\ \ \ &x\in\Omega,\\
\frac{\partial}{\partial n}\phi=0,\ \ \ &x\in\partial\Omega.
\end{array}
\label{vhats}
\right.
\end{equation}
Since $\hat V$ satisfies \eqref{vhat}, we have $\kappa_1(\delta_2, \beta-\mu\hat V)=0$. Hence  $a:=\kappa_1(\delta_2, \beta-2\mu\hat V)<0$, i.e. the principal eigenvalue of \eqref{vhats} is negative. Therefore, $\hat V$ is linearly stable. By the principle of linearized stability, it is locally asymptotically stable. 

Let $\epsilon>0$ be given. Since $\hat V$ is locally asymptotically stable, there exists $\delta>0$ such that $\|V(\cdot, t)-\hat V\|_\infty<\epsilon$ for all $V_0\in C_+(\bar\Omega)$ with $\|V_0-\hat V\|_\infty<\delta$. Let $w(\cdot, t)=V(\cdot, t)-\hat V$. Then $w$ satisfies 
 \begin{equation}
\left\{
\begin{array}{ll} \label{vw}
w_t=\triangledown\cdot \delta_2\triangledown w+ (\beta -2\mu \hat V)w-2\mu w^2,\ \ \ &x\in\Omega, t>0,\\
\frac{\partial}{\partial n}w=0,\ \ \ &x\in\partial\Omega, t>0,\\
w(x, 0)=V_0-\hat V,\ \ \ &x\in\Omega.
\end{array}
\right.
\end{equation}
Let $S(t)$ be the semigroup generated by $\triangledown\cdot \delta_2\triangledown+ (\beta -2\mu \hat V)$ (associated with Neumann boundary condition) in $C(\bar\Omega)$. Then there exists $M_1>0$ such that $\|S(t)\|\le M_1e^{-at}$ for all $t\ge 0$. Then by \eqref{vw}, we have
$$
w(\cdot, t)=S(t)w(\cdot, 0) - \int_0^t S(t-s)\mu w(\cdot, s)^2 ds.
$$
It then follows that
\begin{eqnarray*}
\|w(\cdot, t)\|_\infty &\le& \|S(t)w(\cdot, 0)\|_\infty +\int_0^t \|S(t-s)\mu w(\cdot, s)^2\|_\infty ds\\
&\le&M_1e^{-at}\|w(\cdot, 0)\|_\infty + \epsilon M_1\|\mu\|_\infty \int_0^t e^{-a(t-s)}\|w(\cdot, t)\|_\infty ds.
\end{eqnarray*}
By the Gronwall's inequality, if $\epsilon\le a/2\|\mu\|_\infty M_1$, we have
$$
\|w(\cdot, t)\|_\infty\le M_1\|V_0-\hat V\|_\infty e^{(M_1\|\mu\|_\infty\epsilon-a)t}\le M_1\|V_0-\hat V\|_\infty e^{-at/2}.
$$
Therefore, $\hat V$ is exponentially asymptotically stable. 
\end{proof}

We then prove the local stability of $E_1$ when $R_0=1$.
\begin{lemma}\label{lemma_locals}
If $R_0=1$, then $E_1$ is locally  stable.
\end{lemma}
\begin{proof}
Let $\epsilon>0$ be given. Denote $V=V_u+V_i$. By Lemma \ref{lemma_exp}, there exists $\delta, M_1, b>0$ such that, if $\|V_{u0}+V_{i0}-\hat V\|_\infty<2\delta$, then 
\begin{equation}\label{exp}
\|V-\hat V\|_\infty\le  M_1\|V_{u0}+V_{i0}-\hat V\|_\infty e^{-bt}.
\end{equation}
Suppose that $(H_{i0}, V_{u0}, V_{i0})$ satisfies $\|H_{i0}\|_\infty\le \delta$, $\|V_{u0}-\hat V\|_\infty\le \delta$ and $\|V_{i0}\|_\infty\le \delta$ such that  \eqref{exp} holds.

Since $\kappa_0$ has the same sign with $R_0-1$, we have $\kappa_0=0$. Let $T(t)$ be the positive semigroup generated by $A=B+C$ in $C(\bar\Omega; \mathbb{R}^2)$. Then there exists $M_2>0$ such that $\|T(t)\|\le M_2$ for all $t\ge 0$. By \eqref{main}-\eqref{maini}, we have
\begin{eqnarray*}
\begin{pmatrix}
H_i(\cdot, t)\\
V_i(\cdot, t)
\end{pmatrix}&=& T(t)
\begin{pmatrix}
H_{i0}\\
V_{i0}
\end{pmatrix}
+\int_0^t T(t-s)
\begin{pmatrix}
0\\
\sigma_2(V_u(\cdot, s)-\hat V)H_i(\cdot, s)-\mu (V(\cdot, s)-\hat V)V_i(\cdot, s)
\end{pmatrix}
ds\\
&\le&T(t)
\begin{pmatrix}
H_{i0}\\
V_{i0}
\end{pmatrix}
+\int_0^t T(t-s)
\begin{pmatrix}
0\\
\sigma_2(V(\cdot, s)-\hat V)H_i(\cdot, s)-\mu (V(\cdot, s)-\hat V)V_i(\cdot, s)
\end{pmatrix}
ds
\end{eqnarray*}

Let $u(t)=\max\{\|H_i(\cdot, t)\|_\infty, \|V_i(\cdot, t)\|_\infty\}$. By \eqref{exp}, we have
\begin{eqnarray*}
u(t)&\le& M_2u(0)+ 2M_2\max\{\|\sigma_2\|_\infty, \|\mu\|_\infty\} \int_0^t \|V_u(\cdot, s)+V_i(\cdot, s)-\hat V\|_\infty u(s) ds\\
&\le&M_2\delta+ \delta C \int_0^t  e^{-bs} u(s) ds
\end{eqnarray*}
where $C=4M_1M_2\max\{\|\sigma_2\|_\infty, \|\mu\|_\infty\}$. Then by Gronwall's inequality, 
\begin{equation}\label{ep1}
u(t)=\max\{\|H_i(\cdot, t)\|_\infty, \|V_i(\cdot, t)\|_\infty\}\le M_2 e^{C\delta/b}\delta.
\end{equation}
Moreover, by \eqref{exp}, we have
\begin{equation}\label{ep2}
\|V_u(\cdot, t)-\hat V\|_\infty \le \|V_u(\cdot, t)+V_i(\cdot, t)-\hat V\|_\infty+ \|V_i(\cdot, t)\|_\infty\le 2M_1\delta+M_2e^{C\delta/b}\delta.
\end{equation}
Combining \eqref{ep1}-\eqref{ep2}, we can find $\delta=\delta(\epsilon)>0$ such that 
$$
\|H_i(\cdot, t)\|_\infty\le \epsilon, \|V_u(\cdot, t)-\hat V\|_\infty\le \epsilon,  \text{ and } \|V_i(\cdot, t)\|_\infty\le \epsilon.
$$
Since $\epsilon>0$ is arbitrary, $E_1$ is locally stable. 
\end{proof}

We then prove the global attractivity of $E_1$ when $R_0=1$.
\begin{theorem}
If $R_0=1$, then $E_1$ is globally stable in the sense that it is locally stable and, for any nonnegative initial data $(H_{i0}, V_{u0}, V_{i0})$ with $ V_{u0}+V_{i0}\neq 0$,
$$
\lim_{t\rightarrow\infty} \|(H_i(\cdot, t), V_u(\cdot, t), V_i(\cdot, t))-E_1\|_\infty=0.
$$
\end{theorem}
\begin{proof}
Let 
$$
\mathbb{M}=\{(H_{i0}, V_{u0}, V_{i0})\in C(\bar\Omega; \mathbb{R}^3_+): \  V_{u0}+V_{i0}=\hat V\}.
$$ 
It suffices to show: (a) $E_1$ is a locally stable steady state of \eqref{main}-\eqref{maini}; (b) The stable set (or basin of attraction) of $E_1$ contains $\mathbb{M}$; (c) The $\omega-$limit set of $(H_{i0}, V_{u0}, V_{i0})$ with $ V_{u0}+V_{i0}\neq 0$ is contained in $\mathbb{M}$. 

By Lemma \ref{lemma_locals}, $E_1$ is locally  stable, which gives (a). If $V_{u0}+V_{i0}\neq 0$, we have $V_u(\cdot, t)+V_i(\cdot, t)\rightarrow \hat V$ in $C(\bar\Omega)$ as $t\rightarrow \infty$, which implies (c).   

To prove (b), suppose $(H_{i0}, V_{u0}, V_{i0})\in \mathbb{M}$. Then the solution of \eqref{main}-\eqref{maini} satisfies $V_u(x, t)+V_i(x, t)=\hat V(x)$ for all $x\in\bar\Omega$ and $t\ge 0$. Hence $(H_i(x, t), V_i(x, t))$ is the solution of the limit problem \eqref{main_reduced}.

Since $R_0=1$, we have $\kappa_0=0$. Let $(\varphi_0, \phi_0)$ be a positive eigenvector associated with $\kappa_0$ of the eigenvalue problem \eqref{eigg}. Motivated by \cite{cui2017dynamics, wu2017dynamics}, for any $w_0:=(H_{i0}, V_{i0})$, we define
$$
c(t; w_0):= \inf\{\tilde c\in\mathbb{R}: H_i(\cdot, t)\le \tilde c\varphi_0 \text{ and } V_i(\cdot, t)\le \tilde c\phi_0\}.
$$
Then $c(t; w_0)>0$ for all $t>0$. We now claim that $c(t; w_0)>0$ is strictly decreasing. To see that, fix $t_0>0$, and we define 
$\bar H_i(x, t)= c(t_0; w_0)\varphi_0(x)$ and $\bar V_i(x, t)=c(t_0; w_0)\phi_0(x)$ for all $t\ge t_0$ and $x\in\bar\Omega$. Then $(\bar H_i(x, t), \bar V_i(x, t))$ satisfies
\begin{equation}\label{compR1}
\left\{
\begin{array}{ll}
\frac{\partial}{\partial t}\bar H_i-\triangledown\cdot \delta_1\triangledown \bar H_i=-\lambda \bar H_i+\sigma_1 H_u\bar V_i,\ \ \ &x\in\Omega, t\ge t_0,\\
\frac{\partial}{\partial t}\bar V_i-\triangledown\cdot \delta_2\triangledown \bar V_i> \sigma_2 (\hat V-\bar V_i)^+\bar H_i-\mu\hat V\bar V_i,\ \ \ &x\in\Omega, t\ge t_0,\\
\frac{\partial }{\partial n}\bar H_i=\frac{\partial}{\partial n}\bar V_i=0,\ \ \ &x\in\partial\Omega, t\ge t_0,\\
\bar H_i(\cdot, t_0)\ge H_i(\cdot, t_0),\  \  \bar V_i(\cdot, t_0)\ge V_i(\cdot, t_0).
\end{array}
\right.
\end{equation}
By the comparison principle for cooperative systems, we have $(\bar H_i(x, t), \bar V_i(x, t))\ge (H_i(x, t), V_i(x, t))$ for all $x\in\bar\Omega$ and $t\ge t_0$. By the second equation of \eqref{compR1}, we get 
$$
\frac{\partial}{\partial t}\bar V_i-\triangledown\cdot \delta_2\triangledown \bar V_i> \sigma_2 (\hat V-\bar V_i)^+ H_i-\mu\hat V\bar V_i.
$$
By the comparison principle, $\bar V_i(x, t)> V_i(x, t)$ for all $x\in\bar\Omega$ and $t> t_0$. Then by the first equation of \eqref{compR1}, 
$$
\frac{\partial}{\partial t}\bar H_i-\triangledown\cdot \delta_1\triangledown \bar H_i>-\lambda \bar H_i+\sigma_1 H_u V_i.
$$
By the comparison principle, we have $\bar H_i(x, t)>H_i(x, t)$ for all $x\in\bar\Omega$ and $t>t_0$. Therefore, $c(t_0; w_0)\varphi_0(x)> H_i(x, t)$ and $c(t_0; w_0)\phi_0(x)> V_i(x, t)$ for all $x\in\bar\Omega$ and $t>t_0$. By the definition of $c(t; w_0)$, $c(t_0; w_0)>c(t; w_0)$ for all $t>t_0$. 
Since $t_0\ge 0$ is arbitrary,  $c(t; w_0)$ is strictly decreasing for $t\ge 0$. 

Let $\tilde \Phi(t)$ be the semiflow induced by the solution of the limit problem \eqref{main_reduced}. Let $\omega:=\omega(w_0)$ be the omega limit set of $w_0$. We claim that $\omega=\{(0, 0)\}$. Assume to the contrary that there exists a nontrivial $w_1\in \omega$. Then there exists $\{t_k\}$ with $t_k\rightarrow\infty$ such that 
$\tilde \Phi(t_k)w_0\rightarrow w_1$. Let $c_*=\lim_{t\rightarrow\infty} c(t; w_0)$. We have $c(t; w_1)=c_*$ for all $t\ge 0$. Actually this follows from the fact that $\tilde\Phi(t)w_1=\tilde\Phi(t)\lim_{t_k\rightarrow\infty}\tilde\Phi(t_k)w_0=\lim_{t_k\rightarrow\infty}\tilde\Phi(t+t_k)w_0$. However since $w_1$ is nontrivial, we can repeat the previous arguments to show that $c(t; w_1)$ is strictly decreasing. This is a contraction. Therefore $\omega=\{(0, 0)\}$, and $(H_i(\cdot, t), V_i(\cdot, t))\rightarrow (0, 0)$ in $C(\bar\Omega)$ as $t\rightarrow\infty$. Since $V_u(\cdot, t)+V_i(\cdot, t)=\hat V$, we have $V_u(\cdot, t)\rightarrow \hat V$ in $C(\bar\Omega)$ as $t\rightarrow 0$. This completes the proof. 
\end{proof}

\section{Concluding remarks}
In this paper, we define a basic reproduction number $R_0$ for the model \eqref{main}-\eqref{maini}, and show that it serves as the threshold value for the global dynamics of the model: If $R_0\le 1$, then disease free equilibrium $E_1$ is globally asymptotically stable; if $R_0>1$, the model has a unique endemic equilibrium $E_2$, which is globally asymptotically stable.  

As shown in Theorem \ref{theorem_ode}, the global dynamics of the corresponding ODE model of \eqref{main}-\eqref{maini} is determined by the magnitude of $\sigma_1\sigma_2H_u/\lambda\mu$. This motivates us to define the local basic reproductive number for model \eqref{main}-\eqref{maini}: 
$$
R(x):=R_1(x)R_2(x)=\dfrac{\sigma_1(x)H_u(x)}{\lambda(x)}\dfrac{\sigma_2(x)}{\mu(x)}.
$$
Since $R_0$ is difficult to visualize, it is natural to ask: are there any connections between $R_0$ and $R$?  As the global dynamics of both models are determined by the magnitude of the basic reproduction number, this is equivalent to ask: how the diffusion rates change the dynamics of the model \eqref{main}-\eqref{maini}, and what is the relation between the reaction-diffusion model \eqref{main}-\eqref{maini} and the corresponding reaction system (the model without diffusion)? We will explore these questions in a forthcoming paper. 
Our main ingredient is the formula: 
$$
R_0=r(L_1R_1L_2R_2)
$$ 
with $L_1:=(\lambda-\triangledown\cdot \delta_1\triangledown)^{-1}\lambda$ and $L_2:=  (\mu\hat V-\triangledown\cdot \delta_2\triangledown)^{-1}\mu\hat V$. This formula establishes an interesting connection between $R_0$ and $R$ as we can prove 
$$r(L_1 L_2)=r(L_1)=r(L_2)=1.
$$ 
Consequences of this formula are:
\begin{enumerate}
\item If $R_i(x)$, $i=1, 2$, is constant, then $R_0=R$; 
\item $R_0>1$ if $R_i(x)>1$, $i=1, 2$, for all $x\in\bar\Omega$ and $R_0<1$ if $R_i(x)<1$, $i=1, 2$, for all $x\in\bar\Omega$. 
\end{enumerate}
Furthermore, when the diffusion coefficients $\delta_1$ and $\delta_2$ are constant, we prove
\begin{itemize}
\item $\lim_{(\delta_1, \delta_2)\rightarrow (\infty, \infty)} R_0= \frac{\int_\Omega\lambda R_1 dx}{\int_\Omega\lambda dx}\frac{\int_\Omega\mu R_2dx}{\int_\Omega\mu dx}$;

\item $\lim_{\delta_1\rightarrow 0}\lim_{\delta_2\rightarrow 0} R_0=\lim_{\delta_2\rightarrow 0}\lim_{\delta_1\rightarrow 0} R_0= \max\{R(x): x\in\bar\Omega\}$.
\end{itemize}

Finally, we remark that our approach is applicable to several other reaction-diffusion models (e.g. \cite{lai2014repulsion, lai2016reaction,  ren2017reaction, pankavich2016mathematical}). For example, the reaction-diffusion within-host  model of viral dynamics studied in \cite{ren2017reaction, pankavich2016mathematical} is
\begin{equation}
\left\{
\begin{array}{lll} \label{mainpan}
&\frac{\partial}{\partial t}T-\triangledown\cdot \delta_1(x)\triangledown T=\lambda(x)-\mu T-k_1TV(-k_2TI),\ \ \ &x\in\Omega, t>0,\\
&\frac{\partial}{\partial t}I-\triangledown\cdot \delta_2(x)\triangledown I=k_1TV (+k_2TI)-\mu_iI\ \ \ &x\in\Omega, t>0,\\
&\frac{\partial}{\partial t}V-\triangledown\cdot \delta_3(x)\triangledown V=N(x)I-\mu_vV,\ \ \ &x\in\Omega, t>0,
\end{array}
\right.
\end{equation}
where $T, I$ and $V$ denote the density of healthy cells, infected cells and virions, respectively. If $\delta_1=\delta_2$ and $\mu=\mu_i$, then $E:=T+I$ satisfies 
$$
\frac{\partial}{\partial t} E-\triangledown\cdot \delta_1(x)\triangledown E=\lambda(x)-\mu E.
$$ 
This equation has a unique positive steady state $\hat E$ and $E(\cdot, t)\rightarrow \hat E$ in $C(\bar\Omega)$ as $t\rightarrow\infty$. Therefore \eqref{mainpan} also has a limit system which is monotone:
\begin{equation*}
\left\{
\begin{array}{lll} 
&\frac{\partial}{\partial t}I-\triangledown\cdot \delta_2(x)\triangledown I=k_1(\hat E-I)^+V (+k_2(\hat E-I)^+I)-\mu_iI\ \ \ &x\in\Omega, t>0,\\
&\frac{\partial}{\partial t}V-\triangledown\cdot \delta_3(x)\triangledown V=N(x)I-\mu_vV,\ \ \ &x\in\Omega, t>0.
\end{array}
\right.
\end{equation*}
For the models in \cite{lai2014repulsion, lai2016reaction}, our method is applicable when there are no chemotaxis. The analysis of the basic reproduction number of all these models can also be done similarly.

\section{Appendix}
 Let $H_i(t), V_u(t)$ and $V_i(t)$ be the density of infected hosts, uninfected vectors, and infected vectors at time $t$ respectively. Then the model is
\begin{equation}
\left\{
\begin{array}{ll}
\frac{d}{d t}H_i(t)=-\lambda H_i(t)+\sigma_1 H_u V_i(t),\ \ \   & t>0,\\
\frac{d}{d t}V_u(t)=-\sigma_2V_u(t) H_i(t)+\beta (V_u(t)+V_i(t))-\mu (V_u(t)+V_i(t))V_u(t),\ \ \ & t>0,\\
\frac{d}{d t}V_i(t)=\sigma_2 V_u(t)H_i(t)-\mu (V_u(t)+V_i(t))V_i(t),\ \ \ & t>0.
\end{array}
\label{1.1}
\right.
\end{equation}
with initial value
$$
(H_i(0),V_u(0),V_i(0)) \in M:=\mathbb{R}_+^3.
$$
The basic reproduction number $R_0$ is defined as
$$
R_0:=\frac{\sigma_1\sigma_2H_u}{\lambda\mu}.
$$
The steady states of \eqref{1.1} are $ss_0=(0, 0, 0)$, $ss_1=(0, \beta/\mu, 0)$, and
\begin{eqnarray*}
ss_2&=&\left(  \frac{\beta(H_u\sigma_1\sigma_2-\lambda\mu)}{\lambda\mu\sigma_2}, \frac{\beta\lambda}{H_u\sigma_1\sigma_2}, \frac{\beta(H_u\sigma_1\sigma_2-\lambda\mu)}{H_u\mu\sigma_1\sigma_2} \right)\\
&=&\left(\frac{\beta(R_0-1)}{\sigma_2}, \frac{\beta}{R_0\mu}, \frac{\lambda\beta(R_0-1)}{H_u\sigma_1\sigma_2}\right)\\
&:=& (\hat H_i, \hat V_u, \hat V_i),
\end{eqnarray*}
which exists if and only if $R_0>1$.

If we add the two last equations of \eqref{1.1} then $N(t):=V_u(t)+V_i(t)$ satisfies the logistic equation
\begin{equation}
\frac{d}{dt} N(t)=\beta N(t)-\mu N^2(t).
\label{1.2}
\end{equation}
We decompose the domain $M:=\mathbb{R}_+^3$ into the partition
$$
M=\partial M_0 \cup M_0,
$$
where 
$$
\partial M_0:=\lbrace (H_i,V_u,V_i) \in M: H_i+V_i=0 \text{ or }  V_u+V_i=0  \rbrace
$$
and
$$
M_0:=\lbrace (H_i,V_u,V_i) \in M: H_i+V_i>0 \text{ and }  V_u+V_i>0  \rbrace=M \setminus \partial M_0.
$$
Biologically, we can interpret $\partial M_0$ as the states without vectors or infected individuals. 
The subregions $\partial M_0$ and $M_0$ are both positively invariant by the semiflow generated (\ref{1.1}).
We can also decompose $M$ with respect to the subdomain
$$
\partial M_1:=\lbrace (H_i,V_u,V_i) \in M: V_u+V_i=0  \rbrace
$$
and
$$
M_1:=\lbrace (H_i,V_u,V_i) \in M: V_u+V_i>0  \rbrace.
$$
Since $N(t):=V_u(t)+V_i(t)$ always satisfies the logistic equation \eqref{1.2},  the subregions $\partial M_1$ and $M_1$ are both positively invariant by the semiflow generated \eqref{1.1}. 
\begin{lemma}\label{lemma_M1}
Both $\partial M_1$ and $M_1$ are positively invariant by the semiflow generated \eqref{1.1}. Moreover,
\begin{enumerate}
\item if $(H_i(0),V_u(0),V_i(0)) \in \partial M_1$, then
\begin{equation*}
\lim_{t\rightarrow\infty} (H_i(t), V_u(t), V_i(t))= (0, 0, 0);
\end{equation*}

\item  if $(H_i(0),V_u(0),V_i(0)) \in M_1$, then
\begin{equation*}
\lim_{t\rightarrow\infty} V_u(t)+ V_i(t)= \frac{\beta}{\mu}.
\end{equation*}
\end{enumerate}
\end{lemma}

If $(H_i(0),V_u(0),V_i(0)) \in M_1$ the long time behavior of \eqref{1.1} is characterized by
\begin{equation}
\left\{
\begin{array}{ll}
\frac{d}{d t}H_i(t)=-\lambda H_i +\sigma_1 H_u V_i,\ \ \   & t>0,\\
\frac{d}{d t}V_i(t)=\sigma_2 (\beta/\mu-V_i)^+H_i-\beta V_i,\ \ \ & t>0,\\
H_i(0)=H_{i0}\ge 0, \ V_i(0)=V_{i0}\ge 0.
\end{array}
\label{reduced}
\right.
\end{equation}

\begin{lemma}\label{lemma_reducedC}
Suppose $R_0>1$. Then \eqref{reduced} has a unique positive steady state $(\hat H_i, \hat V_i)$. Moreover, $(\hat H_i, \hat V_i)$ is locally asymptotically stable, and if $H_{i0}+V_{i0}\neq 0$, then the solution $(H_i, V_i)$ of \eqref{reduced} satisfies 
$$
\lim_{t\rightarrow\infty} (H_i(t), V_i(t))= (\hat H_i, \hat V_i).
$$
\end{lemma}
\begin{proof}
The uniqueness of the positive steady state $(\hat H_i, \hat V_i)$ can be checked directly when $R_0>1$. Let $D=\mathbb{R}_+^2$. Then $D$ is invariant for \eqref{reduced}. It is not hard to show that the solution of \eqref{reduced} is bounded.  

Let $F_1(H_i, V_i)=-\lambda H_i+\sigma_1 H_u V_i$ and $F_2(H_i, V_i)=\sigma_2 (\beta/\mu-V_i)^+H_i-\beta V_i$.  Then $\partial F_1/\partial V_i \ge 0$ and $\partial F_2/\partial H_i \ge 0$ on $D$. So \eqref{reduced} is cooperative. Let $\tilde \Phi(t): D\rightarrow D$ be the semiflow generated by the solution of \eqref{reduced}. Then $\tilde \Phi(t)$ is monotone. 

If $H_{i0}+V_{i0}\neq 0$, then $H_i(t)>0$ and $V_i(t)>0$ for all $t>0$. So without loss of generality, we may assume $H_{i0}> 0$ and $V_{i0}> 0$. We can choose $\delta$ small such that $F_1(\delta \hat H_i, \delta \hat V_i)\ge 0$, $F_2(\delta \hat H_i, \delta \hat V_i)\ge 0$, $H_{i0}\ge \delta \hat H_i$, and $V_{i0}\ge \delta \hat V_i$. By \cite[Proposition 3.2.1]{smith1995monotone}, $\tilde \Phi(t)(\delta \hat H_i, \delta \hat V_i)$ is nondecreasing for $t\ge 0$ and converges to a positive steady state as $t\rightarrow \infty$. Since $(\hat H_i, \hat V_i)$ is the unique positive steady state, we must have $\tilde \Phi(t)(\delta \hat H_i, \delta \hat V_i)\rightarrow (\hat H_i, \hat V_i)$ as $t\rightarrow \infty$. 

Similarly, we may choose $k>0$ such that  $F_1(k \hat H_i, k \hat V_i)\le 0$, $F_2(k \hat H_i, k \hat V_i)\le 0$, $H_{i0}\le k\hat H_i$, and $V_{i0}\le k \hat V_i$. Then $\tilde \Phi(t)(\delta \hat H_i, \delta \hat V_i)$ is non-increasing for $t\ge 0$ and $\tilde \Phi(t)(k \hat H_i, k \hat V_i)\rightarrow (\hat H_i, \hat V_i)$ as $t\rightarrow \infty$. By the monotonicity of $\tilde\Phi(t)$, we have $\tilde\Phi(t)(\delta\hat H_i, \delta \hat V_i)\le \tilde\Phi(t)(H_{i0}, V_{i0})\le \tilde\Phi(t)(k\hat H_i, k \hat V_i)$ for $t\ge 0$. It then follows that $\tilde\Phi(t)(H_{i0}, V_{i0}) \rightarrow (\hat H_i, \hat V_i)$ as $t\rightarrow \infty$.
\end{proof}

 We now present a uniform persistence result.
\begin{lemma}\label{lemma_persistence2} If $R_0>1$, then the semiflow generated by \eqref{1.1} is uniformly persistent with respect to $(M_0, \partial M_0)$ in the sense that there exists $\epsilon>0$ such that, for any $(H_i(0),V_u(0),V_i(0)) \in M_0$, we have
\begin{equation}
\liminf_{t\rightarrow\infty} \inf_{w\in \partial M_0} |(H_i(t), V_u(t), V_i(t))- w|\ge \epsilon.
\end{equation}
\end{lemma}
\begin{proof}
We apply \cite[Theorem 4.1]{hale1989persistence} to prove this result. Let $\Phi(t): \mathbb{R}^3_+\rightarrow \mathbb{R}^3_+$ be the semiflow generated by \eqref{1.1}, i.e. $\Phi(t)w_0=(H_i(t), V_u(t), V_i(t))$ for $t\ge 0$, where $(H_i(t), V_u(t), V_i(t))$ is the solution of \eqref{1.1} with initial condition $w_0=(H_i(0),V_u(0),V_i(0))\in \mathbb{R}_+^3$.

 The semiflow $\Phi(t)$ is point dissipative in the sense that there exists $M>0$ such that $\limsup_{t\rightarrow\infty}\|\Phi(t)w_0\|\le M$ for any $w_0\in \mathbb{R}_+^3$. Actually, Lemma \ref{lemma_M1} implies that $\limsup_{t\rightarrow\infty} V_u(t)\le \beta/\mu$ and $\limsup_{t\rightarrow\infty} V_i(t)\le \beta/\mu$. By the first equation of \eqref{1.1}, we have $\limsup_{t\rightarrow\infty} H_i(t)\le \sigma_1\beta H_u/\mu\lambda$.

We note that $M_0$ and $\partial M_0$ are both invariant with respect to $\Phi(t)$. Moreover, the semiflow $\Phi_\partial (t):=\Phi(t)|_{\partial M_0}$, i.e. the restriction of $\Phi(t)$ on $\partial M_0$, admits a compact global attractor $A_\partial$. If $w_0=(H_i(0),V_u(0),V_i(0))\in \partial M_0$, then the $\omega-$limit set of $w_0$ is $\omega(w_0)=\{ss_0\}$ if $w_0\in M_1$ and $\omega(w_0)=\{ss_1\}$ if $w_0\in M_0\setminus M_1$. Hence we have
$$
\tilde A_\partial:= \cup_{w_0\in A_\partial} \omega(w_0)=\{ss_0\} \cup \{ss_1\}.
$$
This covering is acyclic since $\{ss_1\}\not\rightarrow\{ss_0\}$, i.e. $W^u(ss_1)\cap W^s(ss_0)=\varnothing$. To see this, suppose $w_0=(H_i(0),V_u(0),V_i(0))\in W^u(ss_1)\cap W^s(ss_0)$. By Lemma \ref{lemma_M1}, we have $w_0\in \partial M_1$. Let $(H_i(t),V_u(t),V_i(t))$ be the complete orbit through $w_0$, then $V_u(t)=V_i(t)=0$ for $t\in \mathbb{R}$. So $w_0\not\in W^u(ss_1)$, which is a contradiction.

We then show that $W^s(ss_0)\cap M_0=\varnothing$ and $W^s(ss_1)\cap M_0=\varnothing$. By Lemma \ref{lemma_M1}, $W^s(ss_0)=\partial M_1\subseteq \partial M_0$, and hence $W^s(ss_0)\cap M_0=\varnothing$. To see $W^s(ss_1)\cap M_0=\varnothing$, it suffices to prove that there exists $\epsilon >0$ such that, for any $w_0=(H_i(0),V_u(0),V_i(0))\in M_0$, the following inequality holds
\begin{equation}\label{weak}
\limsup_{t\rightarrow\infty} |\Phi(t)w_0-ss_1|\ge \epsilon.
\end{equation}
Assume to the contrary that \eqref{weak} does not hold. Let $\epsilon_0>0$ be given. Then there exists $w_0\in M_0$ such that 
\begin{equation*}
\limsup_{t\rightarrow\infty} |\Phi(t)w_0-ss_1|< \epsilon_0.
\end{equation*}
So there exists $t_0>0$ such that $\beta/\mu-\epsilon_0\le V_u(t)\le \beta/\mu+\epsilon_0$ and $V_i(t)\le \epsilon_0$ for $t\ge t_0$.

 By \eqref{1.1}, we have
\begin{equation}
\left\{
\begin{array}{ll}
\frac{d}{d t}H_i(t)=-\lambda H_i+\sigma_1 H_u V_i,\ \ \   & t>t_0,\\
\frac{d}{d t}V_i(t)\ge \sigma_2 (\frac{\beta}{\mu}-\epsilon_0)H_i-\mu (\frac{\beta}{\mu}+2\epsilon_0)V_i,\ \ \ & t>t_0.
\end{array}
\label{comparison0}
\right.
\end{equation}
The matrix associated with the right hand side of \eqref{comparison0} is
$$
A_{\epsilon_0}: =
\begin{bmatrix}
-\lambda & \sigma_1H_u\\
\sigma_2(\frac{\beta}{\mu}-\epsilon_0) & -\mu(\frac{\beta}{\mu}+2\epsilon_0)
\end{bmatrix},
$$
whose eigenvalues $\lambda_1$ and $\lambda_2$ satisfy that $\lambda_1+\lambda_2=-\lambda-\mu(\beta/\mu+2\epsilon_0)<0$ and  $\lambda_1\lambda_2=\lambda\mu(\beta/\mu+2\epsilon_0)-\sigma_1\sigma_2H_u(\beta/\mu-\epsilon_0)$. Since $R_0>1$, we can choose $\epsilon_0$ small such that $\lambda_1\lambda_2<0$. Hence either $\lambda_1>0>\lambda_2$ or $\lambda_2>0>\lambda_1$. Without loss of generality, suppose $\lambda_1>0>\lambda_2$. Then by the Perron-Frobenius theorem, there is an eigenvector $(\phi, \psi)$ associated with $\lambda_1$ such that $\phi>0$ and $\psi>0$.

Let $(\tilde H_i(t), \tilde V_i(t))$ be the solution of the following problem
\begin{equation}
\left\{
\begin{array}{ll}
\frac{d}{d t}\tilde H_i(t)=-\lambda\tilde H_i(t)+\sigma_1 H_u\tilde V_i(t),\ \ \   & t>t_0,\\
\frac{d}{d t}\tilde V_i(t)= \sigma_2 (\frac{\beta}{\mu}-\epsilon_0)\tilde H_i(t)-\mu (\frac{\beta}{\mu}+2\epsilon_0)\tilde V_i(t),\ \ \ & t>t_0,\\
\tilde H_i(t_0)=  \delta\phi,  \tilde V_i(t_0)= \delta \psi,
\end{array}
\label{comparison1}
\right.
\end{equation}
where $\delta$ is small such that $H_i(t_0)\ge \tilde H_i(t_0)$ and $V_i(t_0)\ge \tilde V_i(t_0)$. By \eqref{comparison0} and the comparison principle for cooperative systems, we have $(H_i(t), V_i(t))\ge (\tilde H_i(t), \tilde V_i(t))$ for $t\ge t_0$. We can check that the solution of \eqref{comparison1} is $(\tilde H_i(t), \tilde V_i(t))= (\delta\phi e^{\lambda_1 (t-t_0)}, \delta\psi e^{\lambda_1 (t-t_0)})$. It then follows from $\lambda_1>0$ that $\lim_{t\rightarrow\infty} H_i(t)=\infty$ and $\lim_{t\rightarrow\infty} V_i(t)=\infty$, which contradicts the boundedness of the solution.

Our conclusion now just follows from \cite[Theorem 4.1]{hale1989persistence}.
\end{proof}

We now present the result about the global dynamics of \eqref{1.1}.
\begin{theorem} \label{theorem_ode} The following statements hold.
\begin{enumerate}
\item $ss_0$ is unstable;  If $(H_i(0),V_u(0),V_i(0)) \in \partial M_1$, then
$$
\lim_{t\rightarrow\infty} (H_i(t), V_u(t), V_i(t))= ss_0;
$$

\item Suppose $R_0<1$. Then $ss_1$ is globally asymptotically stable, i.e. $ss_1$ is locally asymptotically stable and if $(H_i(0),V_u(0),V_i(0)) \in M_1$, then
$$
\lim_{t\rightarrow\infty} (H_i(t), V_u(t), V_i(t))= ss_1;
$$

\item Suppose $R_0>1$. Then $ss_1$ is unstable, and if  $(H_i(0),V_u(0),V_i(0)) \in \partial M_0\setminus \partial M_1$, then
$$
\lim_{t\rightarrow\infty} (H_i(t), V_u(t), V_i(t))= ss_1. 
$$
Moreover, $ss_2$ is globally asymptotically stable in the sense that $ss_2$ is locally asymptotically stable and  for any $(H_i(0),V_u(0),V_i(0)) \in M_0$,
$$
\lim_{t\rightarrow\infty} (H_i(t), V_u(t), V_i(t)) = ss_2.
$$
\end{enumerate}
\end{theorem}
\begin{proof}
We only  prove the second convergence result in part 3 (see \cite{fitzgibbon2017outbreak} and Lemma \ref{lemma_M1} for the other parts). Since the solution of \eqref{reduced} is bounded, the omega limit set of the solution of \eqref{1.1} exists. 

 Suppose $(H_i(0),V_u(0),V_i(0)) \in M_0$. Then the solution $(H_i(t), V_u(t), V_i(t))$ of \eqref{1.1} satisfies that $H_i(t), V_u(t), V_i(t)>0$ for all $t>0$. Since $V_u(0)+V_i(0)\neq 0$, we have $V_u(\cdot, t)+V_i(\cdot, t)\rightarrow  \beta/\mu$ as $t\rightarrow \infty$. So the limit system of 
 \begin{equation*}
\left\{
\begin{array}{ll}
\frac{d}{d t}H_i(t)=-\lambda H_i+\sigma_1 H_u V_i,\ \ \   & t>0,\\
\frac{d}{d t}V_i(t)=\sigma_2 V_uH_i-\mu (V_u+V_i)V_i,\ \ \ & t>0,
\end{array}
\right.
\end{equation*}
is \eqref{reduced}. By Lemma \ref{lemma_persistence2}, there exists $\epsilon>0$ such that 
$$
\liminf_{t\rightarrow\infty} |H_i(t)|+|V_i(t)|\ge \epsilon. 
$$
Hence the omega limit set of \eqref{1.2} is contained in $W:= \{(H_{i0}, V_{i0})\in R^2_+:  \ H_{i0}+V_{i0}\neq 0\}$. By Lemma \ref{lemma_reducedC}, $W$ is the stable set of the stable equilibrium $(\hat H_i, \hat V_i)$ of \eqref{reduced}. By the theory of  asymptotic autonomous systems, we have $H_i(t)\rightarrow \hat H_i$ and $V_i(t)\rightarrow \hat V_i$ as $t\rightarrow\infty$. Moreover since $V_u(\cdot, t)+V_i(\cdot, t)\rightarrow  \beta/\mu=\hat V_u+\hat V_i$, we have $V_u(t)\rightarrow \hat V_u$ as $t\rightarrow\infty$. 
\end{proof}

\end{document}